\documentclass[ 11pt]{article}

\usepackage{color}
\usepackage{times}

\usepackage{hyperref}

\usepackage{verbatim}
\usepackage{microtype}
\usepackage[utf8]{inputenc}  
\usepackage[T1]{fontenc}       
\usepackage{geometry}        
\usepackage[english]{babel}  
\usepackage{cite}
\usepackage{amssymb}   
\usepackage{float}
\usepackage{amsmath} 
\usepackage{amscd}
\usepackage{graphicx, color}    
\usepackage{amsthm}                 
\usepackage{latexsym}    
\usepackage[all]{xy}

\newtheorem{thm}{Theorem}[section]
\newtheorem{them}{Theorem}
\newtheorem*{thm*}{Theorem}

\newtheorem{cor}[thm]{Corollary}
\newtheorem*{cor*}{Corollary}
\newtheorem{prop}[thm]{Proposition}
\newtheorem*{prop*}{Proposition}
\newtheorem{lem}[thm]{Lemma}

\theoremstyle{definition}
\newtheorem{definition}[thm]{Definition}

\newtheorem{rmk}[thm]{Remark}
\newtheorem*{rmk*}{Remark}
\newtheorem{example}[thm]{Example}

\def\lquotient#1#2{%
\makeatletter
\lower.6ex\hbox{$#1$}\backslash\raise.3ex\hbox{$#2$}%
\makeatother
}

\def\rquotient#1#2{%
\makeatletter
\raise.6ex\hbox{$#1$}/\lower.2ex\hbox{$#2$}%
\makeatother
}

\newcommand{\bbZ}{{\mathbb Z}}

\newcommand{\cC}{{\mathcal C}}

\newcommand{\cE}{{\mathcal E}}

\newcommand{\cH}{{\mathcal H}}

\newcommand{\cK}{{\mathcal K}}
\newcommand{\cL}{{\mathcal L}}

\newcommand{\cP}{{\mathcal P}}

\newcommand{\cR}{{\mathcal R}}

\newcommand{\cW}{{\mathcal W}}

\newcommand{\cY}{{\mathcal Y}}

\newcommand{\fra}{{\mathfrak a}}

\newcommand{\frA}{{\mathfrak A}}

\newcommand{\ra}{\rightarrow}

\newcommand{\hra}{\hookrightarrow}

\DeclareMathOperator{\F}{F}

\title{\textbf{A combination theorem for cubulation in small cancellation theory over free products}}  
\author{A. Martin and M. Steenbock}
\date{}

\begin{document}
\maketitle

\begin{abstract} 

We prove that a group obtained as a quotient of the free product of finitely many cubulable groups by a finite set of relators satisfying the classical $C'(1/6)$--small cancellation condition is cubulable. This yields a new large class of relatively hyperbolic groups that can be cubulated, and constitutes the first instance of a cubulability theorem for relatively hyperbolic groups which does not require any geometric assumption on the peripheral subgroups besides their cubulability. We do this by constructing appropriate wallspace structures for such groups, by combining walls of the free factors with walls coming from the universal cover of an associated $2$-complex of groups. 
\end{abstract}

{\footnotesize
\noindent \textbf{MSC-classification:} 20F06; 20F65; 20F67. 
}

\medskip

\section{Introduction}

The geometry of non-positively curved cube complexes has attracted a lot of attention recently due to the spectacular progress in several related problems, most notably the solution to Thurston's remaining four questions on the structure of $3$-dimensional manifolds, including the virtual Haken conjecture of Waldhausen \cite{AgolVirtualHaken}. An important problem in this circle of ideas is to show that virtually special cubical complexes are \emph{stable} under various geometric operations. 
A main geometric task is to \emph{combine} the various wallspaces at hand to construct a wallspace structure for the group under 
study. 

\subsection{Combination problems}

A general combination problem for wallspaces can be formulated as follows:

\medskip

\noindent \textbf{Combination Problem.}\emph{ Let $G$ be a group acting on a  polyhedral complex $X$ endowed with a wallspace structure, such that each non-trivial face stabiliser admits a wallspace structure. 
\begin{itemize}
 \item Under which conditions can we combine such structures into a wallspace structure for $G$? 
 \item If each stabiliser is cubulable, under which conditions can we ensure that $G$ is cubulable? 
\end{itemize}
}

\medskip

This problem has been extensively studied to combine CAT(0) cube complexes under strong \emph{ (relative) hyperbolicity conditions on the group}:

\paragraph{Amalgams and HNN extensions.}
Haglund-Wise \cite{HaglundWiseCombination} and \cite{HsuWiseAmalgams} prove that virtual specialness of groups is preserved under certain amalgamated products or HNN extensions. In a such setting, $G$ is acting cocompactly on a tree $X$  with vertex stabilisers that are CAT(0) cubulable. Theorem 1.2 of \cite{HaglundWiseCombination} requires the vertex stabilisers and the whole groups to be \emph{Gromov hyperbolic}, while Theorem A of \cite{HsuWiseAmalgams} requires that the group $G$ is hyperbolic \emph{relative to virtually abelian subgroups}. 

\paragraph{Cubical small cancellation theory.}

The malnormal virtually special quotient theorem \cite{wise_structure_2011} deals with cubulable \emph{hyperbolic} groups and proves the cubulability of appropriate \emph{hyperbolic} quotients, under strong \emph{cubical} small cancellation conditions \cite[Def. 5.1]{wise_structure_2011}.

\paragraph{Relatively hyperbolic groups.}

Hruska--Wise \cite{hruska_finiteness_2014} prove that, for a group $G$ that is hyperbolic relative to a finite set of parabolic subgroups $(P_i)$, 
the $G$-action on the CAT(0) cube complex dual to a finite family of relatively quasiconvex subgroups is cocompact \emph{relative to} $P_i$-invariant subcomplexes. If the parabolic subgroups are abelian and if the action on the dual cube complex is proper, they show that the action is cocompact on a \emph{truncation} of that dual cube complex.

\subsection{The main theorem}

The main theorem of this article is a cubulation theorem for large classes of relatively hyperbolic groups, \textit{without any} assumption on the peripheral subgroups besides their cubulability. These groups are realised as $C'(1/6)$ small cancellation groups over free products. Note that finitely presented $C'(1/6)$ small cancellation groups over a free product of groups are hyperbolic relative to their free factors \cite{pankratev_hyperbolic_1999}. 
\begin{them}[cf. Theorem \ref{properaction} and Theorem \ref{cocompactaction}]\label{T: main2} Let $\F$ be the free product of finitely many cubulable groups. If $G$ is a quotient of $\F$ by a finite set of relators which satisfies the classical $C'(1/6)$--small cancellation condition over $\F$, then $G$ is cubulable.
\end{them} 
To the authors' knowledge, there are no prior results for relatively hyperbolic groups to provide a \emph{cocompact} action on the dual CAT(0) cube complex \emph{without} strong  assumptions on the peripheral subgroups. In particular, the cocompactness of the action \emph{does not} assume any condition on the free factors besides their cubulability. This contrasts with the aforementioned previous theorems where either stronger hyperbolic conditions on the group $G$, or stronger conditions on the peripheral subgroups are needed. In particular, we do not need the peripheral subgroups to be hyperbolic, nor do we need that they are virtually abelian.

\paragraph{Small cancellation over free products}
The class of small cancellation groups over free products, both classical and its graphical generalisation, provides a natural setting to study the cubulability of groups acting cocompactly but not properly on higher-dimensional complexes for two reasons. As we explain in this article, such groups act in a very controlled way on $2$-dimensional $C'(1/6)$--polygonal complexes, and therefore provide a manageable framework to develop a good geometric intuition. Moreover, the small cancellation theory over free products allows for the construction of groups with a wide range of algebraic and geometric properties. It was fundamental in showing strong embedding properties of infinite groups \cite{miller_embeddings_1971,schupp_embeddings_1976}, in the solution of non-singular equations over groups \cite{edjvet_nonsingular_2011, edjvet_nonsingular_2010}, in the construction of torsion-free groups without the unique product property \cite{RipsSegevUniqueProduct, SteenbockRipsSegev,ArzhantsevaSteenbockRipsConstruction,
GruberMartinSteenbockFiniteIndex} and in the construction 
of 
acylindrically hyperbolic groups with unexpected properties \cite[Theorem 1.7]{dominik_infintely_2014}.

 \subsection{Comparison to previous work of Wise on cubical small cancellation}\label{S: Comparison}

In the celebrated essay \cite{wise_structure_2011}, Wise outlines a far-reaching extension of his results on the action of finitely presented classical $C'(1/6)$--small cancellation quotients on CAT(0) cube complexes \cite{WiseSmallCancellation}. We explain here how the small cancellation groups over free products considered in this paper can be considered  examples of Wise's cubical small cancellation groups, and to what extent Wise's approach \cite[Th. 5.50, Cor. 5.53]{wise_structure_2011} is sufficient to recover some, but certainly not all, of the results obtained in this paper.

In \emph{this} section, for the sake of a simplified comparison, we use the notations of \cite{wise_structure_2011}.

\paragraph{1. Cubical presentations.}

The general setting of Wise's cubical small cancellation theory deals with so-called \textit{cubical presentations} $\langle X\mid Y \rangle$ \cite[Sec. 3.2]{wise_structure_2011}, 
which consists of a non-positively curved cube complex $X$, and a local isometry $\varphi: Y \ra X$ of non-positively curved cube complexes (Wise's theory deals with an arbitrary number of such maps, but for simplicity we will restrict ourselves to the case of a single local isometry). To such a data, one can associate  its mapping cone $X^*$, whose fundamental group is the quotient of $\pi_1(X)$ by the normal subgroup generated by the image of $\varphi$. 

We obtain a cubical presentation associated with a quotient over a free product of the form $\rquotient{A*B}{\ll w \gg}$, for some appropriate element $w$ of $A*B$, as follows (Here we only treat  the \textit{torsion-free} case.): 
Let us assume that the word $w$ is \emph{not} a proper power, and that $A$ and $B$ are torsion-free. One first constructs a non-positively curved cube complex $X$ with fundamental group $A * B$ by choosing two non-positively curved complexes $X_A$ and $ X_B$ with fundamental groups $A$ and $B$ respectively, and by connecting them by an edge. 
 One can then associate to the word $w$ an immersed simplicial loop $\varphi: P \ra X$. This yields a cubical presentation for the quotient $G:= \rquotient{A*B}{\ll w \gg}$. 
 
\medskip

In Section \ref{SectionSmallCancellation}, we explore  such a construction in a technically precise way. We can then treat such groups  in more generality than we could using only the methods of \cite{wise_structure_2011}.

\paragraph{2. Properness of the action and the generalised $B(6)$--condition.}

 Wise gives conditions of a small cancellation nature so that the universal cover of $X^*$ can be equipped with a wallspace structure that allows for the study of the cubulability and the specialness of $\pi_1(X^*)$.  
In particular, the generalised $B(6)$--condition \cite[Def. 5.1]{wise_structure_2011} is a key ingredient to construct an appropriate wallspace structure for the group in \cite[Th. 5.50]{wise_structure_2011}, and to obtain the properness of the action on the dual CAT(0) cube complex. In presence of strong small cancellation conditions \cite[Th. 3.20, Cor. 3.32]{wise_structure_2011}, Wise shows that the crucial non-positive curvature condition (2) in his generalised $B(6)$--condition holds.

In our previous construction, by choosing a sufficiently large length for the edge joining $X_A$ and $X_B$,  the generalised $B(6)$--condition can be verified for the cubical presentation of $G$. In particular, Wise's work can be adapted to our  setting to show that the groups we consider in this paper act \textit{properly} on a CAT(0) cube complex. Indeed, properness is treated in Theorem 5.50 of \cite{wise_structure_2011}. (One can verify the assertions: The conditions (1),(3),(4),(5) and (6) of Wise generalised $B(6)$--condition are satisfied by construction; Condition (2) follows from Corollary 3.32 (1) in \cite{wise_structure_2011}. The conditions (2),(3), and (4) of \cite[Th. 5.50]{wise_structure_2011} are as well satisfied by construction.) 

With these general ideas of Wise in the background, the approach followed in this article, however, provides a shorter and more transparent explicit proof of the fact that such groups act properly on CAT(0) cube complexes, and does not require the full strength of Wise's machinery. In particular, we do neither use the generalised $B(6)$--condition, nor do we use Wise's detailed analysis of  cubical van Kampen diagrams.

\paragraph{3. Cocompactness of the action.}

Our most important contribution lies in the cocompactness of the action. In Wise's Corollary 5.53 \cite{wise_structure_2011} (and in other related results as mentioned above), cocompactness of the action follows from the hyperbolicity of the quotient group. It is therefore \textit{not} possible to recover our cubulability results from Wise's argument in \cite[Th. 5.50, Cor. 5.53]{wise_structure_2011} when the free factors are not hyperbolic. In contrast to such strong conditions, in this article we control the geometry \textit{of the universal cover} of the complex of groups associated with small cancellation groups over free products of groups, and not necessarily the geometry of the whole group, to understand the  geometric structure of the wallspace. 

\subsection{Complexes of groups}

In this article we adopt the point of view of \emph{complexes of groups}, a high-dimensional generalisation of graph of groups, developed by Gersten--Stallings \cite{GerstenStallings}, Corson \cite{CorsonComplexesofGroups}, and Haefliger \cite{HaefligerOrbihedra}. In particular, we associate to a small cancellation group $G$  over a free product  a $2$-dimensional complex of groups with fundamental group $G$. Its universal cover is a $C'(1/6)$--small cancellation polygonal complex $X$ on which $G$ acts with vertex stabilisers being conjugates of the free factors. To obtain a space quasi-isometric to $G$, we then \emph{blow up} vertices into CAT(0) cube complexes. As a result, we obtain a polyhedral complex with a proper and cocompact $G$-action. It is on such a polyhedral complex that we want to define a wallspace structure, by \textit{combining} the walls in $X$ and the walls of the various cube complexes present in the blown-up space. 

This complex of groups approach is very natural:  It allows us to work directly with the geometric structure of the small cancellation complex $X$. We can use it to explicitly combine walls of the free factors, to obtain a wallspace structure for the small cancellation quotient~$G$.

 It is this complex of groups approach that allows us to remove the strong (relative) hyperbolicity conditions required in aforementioned articles: The polygonal complex $X$ itself is hyperbolic, but the blown-up space, which is quasi-isometric to $G$, can have a very different geometry. One of the key points in this complex of groups approach is to use the geometry of the polygonal complex $X$ to study  the walls constructed in the blown-up space.

Note that our main theorem then follows from the following, slightly more general, statement that can be extracted from our proof of Theorem \ref{T: main2}.

\begin{them}  Let $X$ be $C'(1/6)$--small cancellation polygonal complex on which a group $G$ acts cocompactly, with cubulable vertex stabilisers and trivial edge stabilisers. Then $G$ is cubulable.
\end{them}

\subsection{Applications}
 
The existence of a cubulation, or more generally of a proper action on a CAT(0) cube complex, has many interesting consequences. We list here several corollaries of our main theorem.

 \paragraph{Baum-Connes conjecture.}
 
Recall that a group acting properly on a CAT(0) cube complex has the Haagerup property. In particular, such a group satisfies the strong Baum-Connes conjecture \cite{higson_etheory_2001} and does not have Kazhdan's Property (T). By relaxing our assumptions on the free factors, we obtain a combination theorem for groups acting properly on  \textit{locally finite} CAT(0) cube complexes.
 \begin{them}\label{T: mainHaagerup} Let $\F$ be the free product of finitely many groups acting properly on a locally finite CAT(0) cube complex. If $G$ is the quotient of $\F$ by a finite set of relators which satisfies the classical $C'(1/6)$--small cancellation condition over $\F$, then $G$ acts properly on a locally finite CAT(0) cube complex. In particular, $G$ satisfies the Haagerup property and the strong Baum-Connes conjecture.
 \end{them}

 \paragraph{Consequences of Agol's theorem.}
 
Let us mention two other significant applications of Theorem~\ref{T: main2} in the particular case of (Gromov) hyperbolic groups. By a recent result of Agol \cite{AgolVirtualHaken} building upon a work of Haglund--Wise \cite{HaglundWiseSpecial, HaglundWiseCombination} among others, a hyperbolic group that acts properly and cocompactly on a CAT(0) cube complex is virtually a special subgroup of a right-angled Artin group. In particular, this implies that a cubulable hyperbolic group is residually finite, linear over the integers and has separable quasiconvex subgroups. We thus obtain the following:
\begin{them}
 Let $\F$ be the free product of finitely many hyperbolic cubulable groups. If $G$ is a quotient of $\F$ by a finite set of relators which satisfies the \emph{classical} $C'(1/6)$--small cancellation condition over $\F$, then $G$ is residually finite, linear over the integers and has separable quasiconvex subgroups.
\end{them}

Another application of Agol's theorem, in the context of the Atiyah and Kaplansky zero-divisor conjectures, was provided by \cite{SchreveStrongAtiyah}. The main result therein, based on the work of Linnell--Schick--Okun and collaborators, see for instance \cite{LinnellOkunSchickStrongAtiyah}, implies the Atiyah conjecture on $\ell^2$-Betti numbers for a large class of groups having the Haagerup property, including cubulable hyperbolic groups. We thus obtain the following:
\begin{them}
 Let $\F$ be the free product of finitely many torsion-free hyperbolic cubulable groups. If $G$ is a torsion-free quotient of $\F$ by a finite set of relators which satisfies the classical $C'(1/6)$--small cancellation condition over $\F$, then $G$ satisfies the strong Atiyah conjecture. In particular, $G$ satisfies the Kaplansky zero-divisor conjecture over the complex numbers.
\end{them}

The Kaplansky zero-divisor conjecture asserts that the group ring over the complex numbers of a torsion-free group contains no non-trivial zero-divisor. A usual method to show the Kaplansky conjecture is to prove the unique product property for the group. The question whether small cancellation groups have the unique product property is a difficult and long-standing open problem, cf.  Problem N1140 of Ivanov in \cite{KourovkaProblems2014}. 

 \paragraph{Open problem.}

Torsion-free groups without the unique product property were constructed as \emph{graphical} small cancellation groups over free products \cite{RipsSegevUniqueProduct,SteenbockRipsSegev,ArzhantsevaSteenbockRipsConstruction,GruberMartinSteenbockFiniteIndex}. It is unknown whether these so  called  generalised Rips-Segev groups satisfy the Kaplansky zero-divisor conjecture. It is therefore  natural to ask, in light of Agol's theorem, whether our approach can be extended to cubulate some generalised Rips--Segev groups.

\subsection{Methods}

Let us detail the idea and structure of our proof. 

\paragraph{Complexes of groups and spaces.}

In Section $2$, we realise a $C'(1/6)$ small cancellation group $G$ over the free product $\F$ as the fundamental group of a developable $2$-dimensional complex of groups, the universal cover of which is a $C'(1/6)$--small cancellation polygonal complex. \emph{From now on we denote this polygonal complex by $X$.} In order to prove that a group is cubulable, a useful approach---which goes back to ideas of Sageev \cite{SageevCubeComplex, HaglundPaulinWallSpaces}---is to define an appropriate wallspace structure on it. Therefore, we first want a space with a proper and cocompact action of $G$. The polygonal complex $X$ does not have this property in general. Indeed, vertex stabilisers are conjugates of (the image in $G$ of) the possibly infinite free factors of the free product $\F$.

\paragraph{The blow up space.}

 To overcome this problem, we \textit{blow up} vertices of $X$. More precisely, we construct a simply connected space $\cE G$ with a proper and cocompact $G$-action as a \textit{complex of spaces} (a high-
dimensional generalisation of the notion of tree of spaces) over $X$. This complex has a polyhedral structure and is a union of CAT(0) cube complexes and polygons. The CAT(0) cube complexes are exactly the preimages of vertices of $X$ and each one is endowed with a geometric action by the associate vertex stabiliser. The polygons of $\cE G$ are in one-to-one 
correspondence with the polygons of $X$; some of their edges map homeomorphically to edges of $X$, while portions of their boundary are   geodesics in some of the  CAT(0) cube complexes contained in $\cE G$ (see Figure \ref{FigureEG}). This construction can be thought as a generalisation of the action of a classical $C'(1/6)$--small cancellation quotient over the \emph{free group} on the universal cover of its presentation complex.

\paragraph{Walls on the building blocks of $\cE G$.}

The space $\cE G$ is built up from $X$ and the fibre CAT(0) cube complexes.

In Section $3$, we put a wallspace structure on (the set of vertices of) $\cE G$. First notice that  the walls of the small cancellation complex $X$, the so-called \textit{hypergraphs} introduced by Wise \cite{WiseSmallCancellation}, naturally lift to walls of $\cE G$. In the case where one of the free factors in the free product $\F$ is infinite however, this collection of walls is not enough to separate elements of $G$ in a conjugate of the image of that factor. This corresponds to the problem of separating vertices of $\cE G$ in one of the CAT(0) cube complexes which is the preimage of a vertex of $X$ with an infinite stabiliser. Nonetheless, vertices of a CAT(0) cube complex are separated by so-called hyperplanes. We therefore want to ``extend'' hyperplanes in a given CAT(0) cube complex to walls of the whole space $\cE G$. In order to do that, we extend Wise's approach \cite{WiseSmallCancellation, wise_structure_2011} to this more general setting.

\paragraph{Walls on complexes of CAT(0) cube complexes.}

Namely, every time a polygon $\widetilde{R}$ of $\cE G$ 
crosses a hyperplane in some vertex fibre along an edge $e$, we  want to combine this hyperplane with the diameter of $\widetilde{R}$ (seen as a wall) starting at the midpoint of $e$. Such a procedure should have the feature that the resulting walls should be realised as trees of hyperplanes over \textit{generalised hypergraphs} of $X$. However, since polygons of $\cE G$ have part of their boundary contained in the vertex fibres, the overlaps between polygons of $\cE G$ can be quite different from the well controlled overlaps between polygons of the small cancellation complex $X$. In order to overcome this problem, we first perform an appropriate subdivision, called ``balancing'', of the complex  (see Definition \ref{balanced} for a precise definition). This procedure, as well as the construction of walls, is detailed in Sections $2.2$, $2.3$ and $2.4$. The aforementioned generalised hypergraphs of $X$, together with the associated generalised \emph{hypercarriers}, are introduced in Section $2.1$. They enjoy 
the same properties as the usual 
notions introduced 
in \cite{WiseSmallCancellation}, and Wise's argument extends to this more general setting in a straightforward way; we give the full proofs of these results in an Appendix.

\paragraph{Properness and cocompactness.}

Finally, we study in Section $4$ the set of walls of $\cE G$.  Namely, we prove that this set of walls satisfies criteria which, as shown by Chatterji--Niblo \cite{ChatterjiNibloWallSpaces}, imply that the action of $G$ on the CAT(0) cube complex associated with the wallspace structure is proper and cocompact. This concludes the proof of Theorem \ref{T: main2}.

\subsection{Acknowledgements} 

We wish to express our gratitude to Goulnara Arzhantseva for her encouragement to work on this problem as well as for many related discussions. Moreover, we thank Fr\'ed\'eric Haglund and  Piotr Przytycki for their interest and comments on a previous version of this work.  We also wish to thank an anonymous referee whose comments helped us improve our comparison with Wise's cubical small cancellation theory, and our introduction, in a previous version.

Both authors were partially supported by the European Research Council (ERC) grant of Goulnara Arzhantseva, grant agreement n$^o$ 259527. The second author is recipient of the DOC fellowship of the Austrian Academy of Sciences and was partially supported by the University of Vienna Research Grant 2013. 

Both authors acknowledge the support of the Erwin Schr\"odinger International Institute in Vienna for hosting the workshop ``Geometry of computation in groups'', during which part of this research was conducted.

\section{Complexes of groups and small cancellation over  free products of groups} \label{SectionSmallCancellation}

Suppose $G$ is a finitely presented group, viewed as a \emph{quotient of the free group $\mathbb{F}_n$} on $n$ generators. That is, $G$ is given by generators $g_1,\ldots,g_n$ and finitely many relators $r_1,\ldots,r_m\in \mathbb{F}_n$ such that $G$ is the quotient of the free group by the normal closure of the subgroup generated by the relators. We now recall the constructions of the   \emph{presentation complex} and the \emph{Cayley complex} associated with such presentations. We start from the bouquet of $n$ oriented cycles $c_1,\ldots c_n$. We label each cycle $c_i$ by the generator $g_i$. For each relator $r_j$ we take a polygonal 2-cell $R_j$, whose boundary edges are oriented and labelled by the generators such that the label of a boundary path of $R_j$ equals $r_j$. Then glue $R_j$ to the bouquet along its boundary word. The complex so obtained is the \emph{presentation complex of $G$}. Its universal cover is the \emph{Cayley complex of $G$}. Note that the fundamental group of the 
presentation complex is $G$, and $G$ has a free and cocompact action on the associated Cayley complex.  

In this paper, we are interested in properties of groups $G$ that are quotients of the (non-trivial) \emph{free product} $\F$ of finitely many groups. 
In this section, we associate to a small cancellation quotient $G$ of the free product of two groups a developable $2$-dimensional complex of groups with fundamental group $G$, the universal cover of which is a small cancellation polygonal complex, see Definition \ref{smallcancellationcomplex}. We shall think about this complex of groups as of an analogue for the presentation complex in the case of quotients of \emph{free product of groups}. The action of $G$ on the universal cover is no longer proper as soon as one of the free factors is infinite. More precisely, stabilisers of vertices correspond to conjugates of the free factors in $G$.  However, we can construct another polyhedral complex with a proper and cocompact $G$-action, by \emph{blowing up} vertices of the universal cover. This polyhedral complex is the  analogue of the Cayley complex for quotients of free products of groups, and is obtained as a \emph{complex of spaces} over the universal cover.

\subsection{Small cancellation groups over  free products of  groups}

We summarize some aspects of the small cancellation theory of the free product of two groups. A more complete treatment can be found in \cite[Chapter V.9]{LyndonSchupp} or \cite[Ch. 11]{olshanskii_geometry_1991}. 
We let $\F=A*B$ be the free product of two groups $A$ and $B$. The groups $A$ and $B$ are called the \emph{free factors}. 
 Every non-trivial element of $\F$ can be represented in a unique way  as a product $w=h_1\cdots h_n$, called the \emph{normal form}, where  $h_i$ is a non-trivial element   in either $ A$ or  $B$ and no two consecutive $h_i$, $h_{i+1}$ belong to the same free factor. Then the \emph{free product length} of $w$ is given by $|w|:=n$. 

The normal form of $w$ is \emph{weakly cyclically reduced} if $|w|\leqslant 1$ or $h_1\not = h_n^{-1}$. 
 If $u,v \in \F$, $u=h_1\cdots h_n$, $v=k_1\cdots k_m$,  and $h_n=k_{1}^{-1}$, then $h_n$ and $k_{1}$ \emph{cancel} in the product $uv$. Otherwise, we say that the product $uv$ is \emph{weakly reduced}.

Let $\cR\subset \F$ be a subset of $\F$, each element of which is represented by  a weakly cyclically reduced normal form.  
Let $G$ be the group defined as
\[
G:= \rquotient{\F}{\ll \cR \gg},
\]
where $\ll \cR \gg$ denotes the normal closure of $\cR$ in $\F$. We say that $\cR$ is \textit{symmetrised} if it is stable by taking  cyclic conjugates and inverses. Up to adding all cyclic conjugates of elements of $\cR$ and their inverses, we can always assume that $\cR$ is symmetrised.

An element $p$ in $\F$ is a \emph{piece} if there are distinct relators $r_1,r_2\in \cR$ such that the products $r_1=pu_1$ and $r_2=pu_2$ are weakly reduced.

The set $\cR$ satisfies the $C'(1/6)$--\emph{condition (over $\F$)} if it is symmetrised and if for every piece $p$ and every relator $r\in \cR$ such that the product $r=pu$ is weakly reduced, we have that \[|p|<\frac16 |r|.\]
In this case we say that $G$ is a $C'(1/6)$--\emph{group (over $\F$)}.

\begin{thm}[Corollary V.9.4 of \cite{LyndonSchupp}]\label{T: local factor injections}
Let $G$ be a $C'(1/6)$--group over the free product $\F$. Then the projection map $\F \ra G$ embeds each free factor of $\F$. \qed
\end{thm}
 
\begin{thm}[cf. \cite{pankratev_hyperbolic_1999}]
 Let $G$ be a $C'(1/6)$--group over the free product $\F$. Then $G$ is hyperbolic relative to the free factors.  If all free factors are hyperbolic, then so is $G$.
\end{thm}

\begin{example}[Fuchsian groups] \label{E: Fuchsian} Fuchsian groups are the fundamental groups of surfaces of genus $g$ with $r$ cone-points of order $m_1, \ldots  m_r$, and $s$ points or closed discs removed. They are generated by $a_1,\ldots,a_g, b_1, \ldots, b_g, x_1,\ldots,x_r,y_1,\ldots,y_s$, with the relators $$\Pi_{i=1}^g [a_i,b_i]\Pi_{j=1}^r x_j\Pi_{k=1}^s y_k,x_1^{m_1},\ldots,x_r^{m_r}.$$
If $4g+r+s+t\geq 6$, then the set of relators obtained from symmetrising the word \[\Pi_{i=1}^g [a_i,b_i]\Pi_{j=1}^r x_j\Pi_{k=1}^s y_k\] satisfies the $C'(1/6)$--condition over the free product 
\[
\langle a_1\rangle * \langle b_1\rangle* \cdots *\langle a_g \rangle * \langle b_g\rangle * \langle x_1\mid x_1^{m_1} \rangle *\cdots * \langle x_r\mid x_r^{m_r}\rangle * \langle y_1 \rangle * \cdots\langle y_s \rangle.
 \] 
\end{example}

\subsection{Complex of groups associated with $C'(1/6)$--groups over a free product of groups}\label{complexofgroupsfreeproduct}

Let $w\in \F$ be an element satisfying the $C'(1/6)$--condition over $\F=A\ast B$ and define the group  
 \[
 G:=\rquotient{A*B}{ \ll w\gg }.
 \]                                                                                            
Observe that $w$ acts hyperbolically on the Bass-Serre tree associated with $A*B$ by the small cancellation condition, and thus we can write 
\[
 w=(a_0b_0\ldots a_{N-1}b_{N-1})^d,
\]
 where $d \geq 1$ and $a_0b_0\ldots a_{N-1}b_{N-1}$ is not a proper power in  $A *B$.  The  theory that we develop in this paper can readily be extended to the free product of finitely many groups and to quotients with respect to finitely many relators. 
 
  We now construct a complex of groups whose fundamental group is $G$.   
 We start by defining several complexes, see Figure \ref{FigureComplexGroups}. 

\begin{itemize}
\item Let $L$ be the simplicial complex consisting of a single edge with vertices $u_A$ and $u_B$, and let $L'$ be its barycentric subdivision with $c$ being the barycentre of $L$. The space $L'$ consists of two edges $e_A$ (containing $u_A$) and $e_B$ (containing $u_B$).
\item Let $R_0$ be the \textit{model polygon} on $2N$ sides, that is, a polygonal complex consisting of  a single $2$-cell whose boundary consists of $2N$ edges. We choose an orientation of $R_0$, a vertex $v_0$ in $\partial R_0$, and then denote by $(v_i)_{i \in \bbZ / 2N \bbZ}$ the remaining vertices, so that, seen from $v_i$, the vertex  $v_{i+1}$ is the next vertex in the  positive direction on $\partial R_0$.
 \item Let $R_{0,\mathrm{simpl}}$ be the simplicial complex obtained from $R_0$ by adding a vertex, called \textit{apex}, in the centre of the $2$-cell, and, for each vertex $v_i$  an edge, called \textit{radius}, joining the apex~to~$v_i$.
  In particular, $R_{0,simpl}$ is the simplicial cone over a loop on $2N$ edges.  
\item Let $R_{0,\mathrm{simpl}}'$ be the  the barycentric subdivision of $R_{0,\mathrm{simpl}}$.  
\end{itemize}

Let us \emph{orient the edges} in the $1$-skeleton of $L'$ and $R_{0,simpl}'$  as follows. If $\sigma \varsubsetneq \sigma'$ are two simplices of $L$ or $R$ (i.e. vertices, edges, or faces) with barycentres $c$ and $c'$ respectively, then the edge between $c$ and $c'$ is oriented from $c'$ to $c$; the barycentre $c'$ is called the \emph{initial vertex} of that edge, the barycentre $c$ is called the \emph{terminal vertex} of that edge. The two edges of $L'$ are, in particular, oriented towards the vertices $u_A$ and $u_B$ respectively. If $\sigma \varsubsetneq \sigma' \varsubsetneq \sigma''$ with barycentres $c$, $c'$ and $c''$ respectively, then the edges $a$ from $c''$ to $c'$ and $b$ from $c'$ to $c$ are said to be \emph{composable}, and their composition is defined to be the edge from $c''$ to $c$, which we denote $ba$. 

Starting from these complexes, we now define the CW-complexes 
\[
 K := \rquotient{( L \sqcup R_0 )}{\simeq}, ~~K_\mathrm{simpl} := \rquotient{( L \sqcup R_{0,\mathrm{simpl}} )}{\simeq} ~\text{  and  }~ K_\mathrm{simpl}' := \rquotient{( L' \sqcup R_{0,\mathrm{simpl}}' )}{\simeq}.
\]

Let us first describe $K_\mathrm{simpl}'$. Here we identify oriented edges in the boundary of $R_{0,\mathrm{simpl}}'$ pointing towards a vertex $v_{2i}$ with the oriented edge $e_A$ of $L'$, while oriented edges in the boundary of $R_{0,\mathrm{simpl}}'$ pointing towards a vertex $v_{2i+1}$ are identified with the oriented edge $e_B$. The resulting simplicial complex is $K_\mathrm{simpl}'$. 
The construction is illustrated in Figure \ref{FigureComplexGroups}. 
 Now, let 
 \[
 q:L' \sqcup R_{0,\mathrm{simpl}}'\to K_{simpl}' 
 \]
 denote the projection, seen as the map between the underlying topological spaces. The map $q$ restricts to a homeomorphism on the interior of each cell of $L \sqcup R_0$ and $L \sqcup R_{0,\mathrm{simpl}}$. We can therefore push forward the CW-structures of $L \sqcup R_0$ and $L \sqcup R_{0,\mathrm{simpl}}$ using the map $q$, and we denote by
$$ K_{simpl} :=q(L \sqcup R_{0,\mathrm{simpl}})\hbox{ and } K:=q( L \sqcup R_0)$$
the associated CW-complexes. In other words, $K_{simpl}$ and $K$ are obtained from $K_{simpl}'$ by forgetting, in each case from left to right, the additional structure we have put on $R_{0,simpl}'$ and $R_{0,simpl}$ respectively. In all three cases, we use apex, radii, and $v_i$ to refer to their respective images in $R_{0,simpl}'$ and $K'_{simpl}$ respectively. \\

A \emph{small category without loop}, or \textit{scwol} in short, is an oriented graph without loop with a notion of \emph{composability} of edges, see  \cite[Chapter~III.$\cC$ Definitions 1.1]{BridsonHaefliger}. The oriented $1$-skeleton of the first barycentric subdivision of a simplicial complex can be endowed with a structure of scwol. In particular, we described a structure of scwol on the oriented $1$-skeleton of $L'$, which we denote $\mathcal{L}'$, and on the oriented $1$-skeleton of $R_{0,simpl}'$. These scwols can be glued together along the map $q$, yielding a structure of scwol on the $1$-skeleton of $K '_{simpl}$, which we denote $\cK '_{simpl}$. 
\\

Observe  that pairs of composable edges of $\cK_\mathrm{simpl}'$ are in $1$-to-$1$ correspondence with triangles of $K_\mathrm{simpl}'$. The simplicial complex $K_\mathrm{simpl}'$ is said to be a \emph{geometric realisation} of the scwol $\cK_\mathrm{simpl}'$. \\

A  \emph{complex of groups} over a scwol $\cY$ consists of the data $(G_\sigma, \psi_\sigma, g_{b,a})$ of \textit{local groups} $G_\sigma$, \textit{local maps} $\psi_\sigma$, and \emph{twisting elements} $g_{b,a}$ for every pair $(b,a)$ of composable edges of $\cY$ subject to additional compatibility conditions, see \cite[Chapter~III.$\cC$, Definition 2.1]{BridsonHaefliger}. To follow our construction details of such kind are not a prerequisite. However, we refer the interested reader to Bridson--Haefliger~\cite[Chapter~III.$\cC$]{BridsonHaefliger} for more terminology and background on complexes of groups.

\begin{definition}\label{complexofgroupsG}
 We define a complex of groups $G(\cK_\mathrm{simpl}')$ over $\cK_\mathrm{simpl}'$ as follows: 
\begin{itemize}
\item the local groups at $u_A$ and $u_B$ are respectively $A$ and $B$, the local group at the apex is $\bbZ / d \bbZ$, and all the other local groups are trivial;
\item all the local maps are trivial;
\item the twisting element associated with a pair of composable edges $(b,a)$, or equivalently to the associated triangle of $K_\mathrm{simpl}'$, is represented in Figure \ref{FigureComplexGroups}. 
\end{itemize}

We now define a morphism $F = (F_\sigma, F(a))$ of complexes of groups from $G(\cK_\mathrm{simpl}')$ to $G$. (A general definition of morphism of complexes of groups can be found in  \cite[Chapter~III.$\cC$, Definition 2.5]{BridsonHaefliger}.)   We first fix some notation.

 For $i=0, \ldots, 2N-1$, let $c_i$ be the barycentre of the radius at the vertex $v_{i}$. For $i=0, \ldots, N-1$, let  $e_i$ be  the oriented edge of  $\cK_\mathrm{simpl}'$ from $c_{2i}$ to $v_{2i}$, and let $f_i$ be the oriented edge of  $\cK_\mathrm{simpl}'$ from $c_{2i+1}$ to $v_{2i+1}$. Let also $e_0'$ be the oriented edge of $\cK_\mathrm{simpl}'$ from $c_0$ to the apex of $R_{0,\mathrm{simpl}}$. 

Now let us set the local maps of $F$ as follows.
\begin{itemize}
\item The local morphisms $F_{u_A}: A \ra G$ and $F_{u_B}: B \ra G $ are the natural projections, the map $\bbZ /d \bbZ \ra G$ sends the generator $\overline{1}$ of $\bbZ /d \bbZ$ to the image of $a_0b_0\ldots a_{N-1}b_{N-1}$ in $G$, and the other maps are trivial;
\item For $i=0, \ldots, N-1$, we set $F(e_i):= a_i$, $F(f_i):=b_i$, $F(e_0'):= \overline{1}$, and $F$ is trivial on all the other edges of $\cK_\mathrm{simpl}'$. 
\end{itemize}
\end{definition}

\begin{figure}[H]
\begin{center}
\scalebox{0.9}{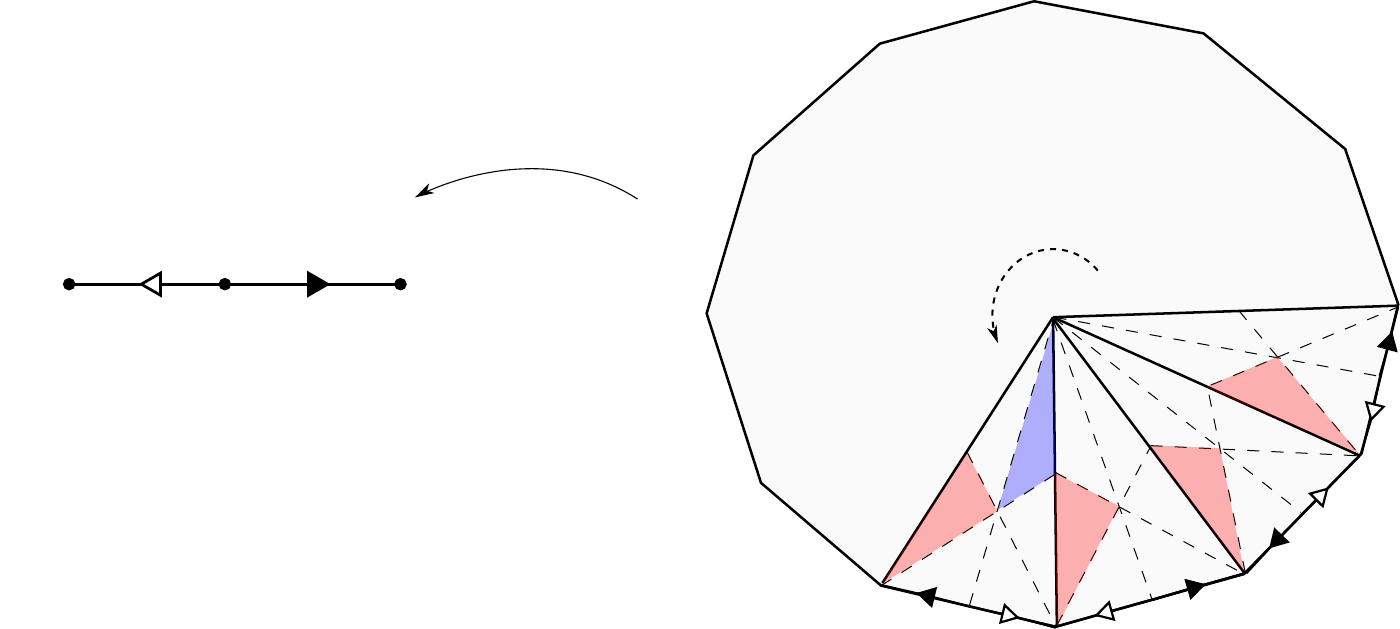}
\caption{Part of the complex of groups $G(\cK_\mathrm{simpl}')$. Twisting elements corresponding to white triangles of $R_{0, \mathrm{simpl}}'$ are trivial. (The element $\overline{1}$ denotes the generator of $\bbZ/d\bbZ$.)}
\label{FigureComplexGroups}
\end{center}
\end{figure}

Let $\pi_1(G(\cK_\mathrm{simpl}'), u_A)$ be the fundamental group of $G(\cK_\mathrm{simpl}')$ at the vertex $u_A$, seen as the group of homotopy classes of $G(\cK_\mathrm{simpl}')$-loops, see \cite[Chapter~III.$\cC$, Definition 3.5]{BridsonHaefliger}. Let $\pi_1F: \pi_1(G(\cK_\mathrm{simpl}'), u_A) \ra  G$ be the associated morphism of fundamental groups, see \cite[Chapter~III.$\cC$, Proposition 3.6]{BridsonHaefliger}. The following result is not surprising when viewed against the aforementioned construction of the presentation complex. However, as complexes of groups are in some technical points surprisingly different to the standard situation, we give an elementary proof using the language of \cite[Chapter~III.$\cC$ Section 3]{BridsonHaefliger}.

\begin{prop}\label{isompi1}
The map 
\[
\pi_1F: \pi_1(G(\cK_\mathrm{simpl}'), u_A) \ra  G
\]
 is an isomorphism.
\end{prop}

\begin{proof}
Since $A$ and $B$ generate $A * B$, and thus $G$, the map $\pi_1F$ is surjective. Let $g$ be an element of $\ker \pi_1F \subseteq \pi_1(G(\cK_\mathrm{simpl}'), u_A)$, and let $\gamma$ be a $G(\cK_\mathrm{simpl}')$-loop based at $u_A$ in the homotopy class $g$. Note that it is possible to homotop $\gamma$ to a loop the support of which is contained in the image of $L'$ in $K_\mathrm{simpl}'$. 

In other words, if we denote by $i: G(\cL') \ra G(\cK_\mathrm{simpl}')$ the natural embedding of complexes of groups (that is, the pullback of $G(\cK_\mathrm{simpl}')$ under the inclusion of scwols $\cL' \hra \cK_\mathrm{simpl}'$), then the induced morphism of fundamental groups  $\pi_1i: \pi_1(G(\cL'), u_A) \ra \pi_1(G(\cK_\mathrm{simpl}'), u_A) $ is surjective. Let $h$ be an element of $ \pi_1(G(\cL'), u_A)$ such that $g= \pi_1i(h)$. We thus have $\pi_1F (\pi_1i(h))= 0$. But since $\pi_1F \circ \pi_1i: \pi_1(G(\cL'), u_A) \ra G$ is the natural projection $A*B \ra \rquotient{A*B}{\ll w\gg}$, it follows that $h$ is in the normal 
subgroup generated by the $G(\cL')$-loop $(a_0, e_A^{-1}, e_B, b_0, e_B^{-1}, e_A, a_1, \ldots)^d$.
Thus, $g$ is in the normal closure of the $G(\cK_\mathrm{simpl}')$-loop $(a_0, e_A^{-1}, e_B, b_0, e_B^{-1}, e_A, a_1, \ldots)^d$. It is now enough to prove that such a $G(\cK_\mathrm{simpl}')$-loop is homotopically trivial. But the definition of $\pi_1(G(\cK_\mathrm{simpl}'), u_A)$  implies that this loop is homotopic to the following edge-path (seen as a $\pi_1(G(\cK_\mathrm{simpl}'), u_A)$-loop): 

\begin{figure}[H]
\begin{center}
\scalebox{0.7}{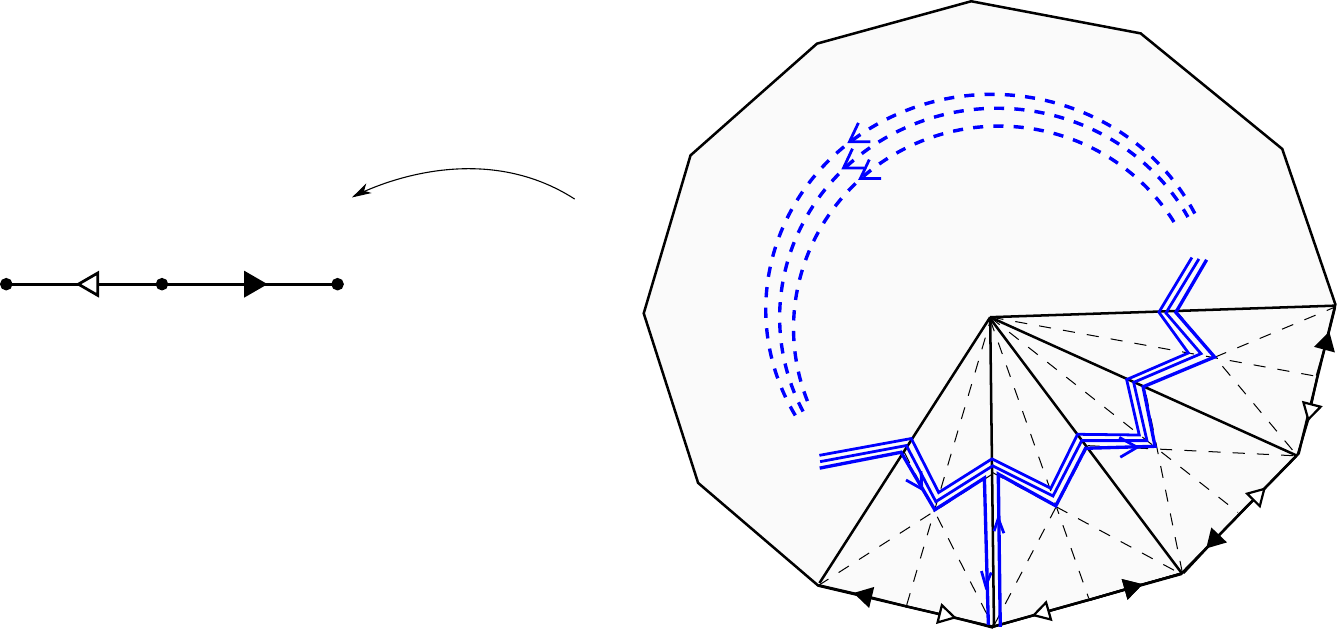}
\caption{A homotopically trivial $G(\cK_\mathrm{simpl}')$-loop.}
\label{FigureHomotopyLoop}
\end{center}
\end{figure}

which is homotopically trivial since the local group at the apex is $\bbZ / d \bbZ$, hence the result.
\end{proof}

Let $\frA^{(0)}(\cK_\mathrm{simpl}')$ be the set of vertices of $\cK_\mathrm{simpl}'$, let  $\frA^{(1)}(\cK_\mathrm{simpl}')$ be the set of edges of $\cK_\mathrm{simpl}'$, and  let $\frA^{(2)}(\cK_\mathrm{simpl}')$ denote the set of pairs $\fra= (a_2, a_1)$ of composable edges of $\cK_\mathrm{simpl}'$.  For every (oriented) edge $a$ define $i(a)$ to be the initial vertex, and $t(a)$ to be the terminal vertex. For $\fra=(a_2, a_1) \in A^{(2)}(\cY)$, we set $i(\fra):= i(a_1)$ and $t(\fra):=t(a_2)$.  
We define maps 
\[
\partial_0,\partial_1:\frA^{(1)}\to \frA^{(0)}
\]
 by setting $\partial_0 (a) := i(a)$ and $\partial_1(a) := t(a)$. 
For $0 \leq i \leq 2$, we define maps 
\[
\partial_i: \frA^{(2)}(\cK_\mathrm{simpl}') \ra \frA^{(1)}(\cK_\mathrm{simpl}')
\] by setting   
$ \partial_0 (a_2, a_1) :=  a_ 2,$ 
$  \partial_1(a_2,a_1) := a_{2}a_1,$ and 
 $\partial_2(a_2, a_1):= a_1. $

Let $\Delta^k$ be the standard Euclidean $k$-simplex, that is, the set of elements $(t_0, \ldots, t_k)$ with $ t_i \geq 0$ and $\sum_i t_i=1$. For $k \geq 1$ and $0 \leq i \leq k$, we denote the embeddings of the sides of $\Delta^k$ by 
\[
d_i: \Delta^{k-1} \ra \Delta^k,\]
 defined by sending $(t_0, \ldots, t_{k-1})$ to $(t_0, \ldots, t_{i-1},0,t_i, \ldots, t_{k-1})$. 

Since the morphism $F$ is injective on the local groups, we can define the following complex. 
\begin{definition}
Let $X_{\mathrm{simpl}}'$ be the simplicial complex obtained from the disjoint union 
\[
\underset{0 \leq k \leq 2}{\coprod}~ \underset{\fra \in \frA^{(k)}(\cK_{simpl}')}{\coprod}
 \bigg( \lquotient{F_{ i(\fra)}(G_{i(\fra)})}{G} \times \{\fra\} \times \Delta^{k} \bigg)
\]
by identifying pairs of the form 
\[
( [gF(a)^{-1}],\partial_i \fra, x ) \hbox{ and }\big([g], \fra, d_i(x)\big)\hbox{ for 
$0 \leq i \leq k$
,}
\]
 where $a$ denotes the edge with initial vertex $i(\fra)$ and terminal vertex $i(\partial_i \fra)$. 
\end{definition}
Note that there is a natural projection   
\[
\pi: X'_{simpl}\to K'_{simpl} 
\]
obtained by forgetting the first coordinate. The CW-structure on $K_{simpl}$ can be pulled-back along $\pi$, yielding a simplicial complex $X_{\mathrm{simpl}}$ with barycentric subdivision $X_{simpl}'$. For simplicity reasons, we still denote by $\pi$ the projection map $ X_{\mathrm{simpl}} \ra K_{\mathrm{simpl}}$.

\medskip

We now construct our \emph{polygonal complex $X$} as the pull back of the CW-structure on $K$ along $\pi$. We can obtain $X$ from $X_{simpl}$ as follows. We denote by $s\in K_{\mathrm{simpl}}$ the apex of $K_{\mathrm{simpl}}$ and by $S \subset X_{\mathrm{simpl}}$ the preimage of $s$ under the projection $\pi:X_{\mathrm{simpl}} \ra K_{\mathrm{simpl}}$, called the \textit{set of apices} of $X_{\mathrm{simpl}}$. A \textit{simplicial polygon} of $X_{\mathrm{simpl}}$ is the star in $X_{\mathrm{simpl}}$ of an apex of $S$, that is, the subcomplex consisting of all simplices containing that apex as a vertex. Two distinct simplicial polygons of $X_{\mathrm{simpl}}$ are either disjoint or meet along a subset of $\pi^{-1}(L)$.
 Let us delete all the apices of $X_{\mathrm{simpl}}$ and all the edges containing them to obtain a polygonal complex denoted $X$, that is, a CW-complex such that $2$-cells are modelled after a \textit{model polygon} $\widetilde{R}_0$ on $d\cdot2N$ sides (which is an orbifold cover of the model polygon $R_0$ on $2N$ sides), 
and such that the various gluing maps $\partial \widetilde{R}_0 \ra  \pi^{-1}(L)$ are simplicial. Furthermore, we identify $\widetilde{R}_0$ with the polygon of $X$ whose apex in $X_{simpl}$ corresponds to the point $ \{1\} \times \{s\} \times \Delta^0$. 

\medskip

By definition, $X'_{simpl}$ is the geometric realisation of the development $D(\cK_{simpl}',F)$, see \cite[Chapter~III.$\cC$, Theorem  2.13]{BridsonHaefliger}. Note that the following result on complexes of groups follows directly from {\cite[Chapter~III.$\cC$, Proposition 3.14]{BridsonHaefliger}}.

\begin{prop}
Let $G(\cY)$ be a complex of groups over a scwol $\cY$ whose geometric realisation is a simplicial complex, $v$ be a vertex of $\cY$ and $F:G(\cY) \ra G$ a morphism from $G(\cY)$ to some group $G$ that is injective on the local groups.

The geometric realisation of the development $D(\cY, F)$ is a universal cover of the complex of groups $G(\cY)$ if and only if the induced morphism $\pi_1F:\pi_1(G(\cY), v) \ra G$ is an isomorphism.\qed
\label{propuniversalcover}
\end{prop}

We thus obtain the following.

\begin{prop}  The simplicial complex $X_{simpl}'$ is (equivariantly isomorphic to) a universal cover of $G(\cK_{simpl}')$. In particular, the small cancellation group $G$ acts on $X_{simpl}'$ with quotient $K'_{simpl}$, with vertex stabilisers $A$, $B$, or $\mathbb{Z}/d\mathbb{Z}$ at vertices mapped under $\pi$ on the vertices $u_A$, $u_B$, or the apex respectively, and  with trivial edge stabilisers. 
\end{prop}

\begin{proof}
 It is enough to prove that the conditions of Proposition \ref{propuniversalcover} are satisfied. The geometric realisation of $\cK_{simpl}'$ is the simplicial complex $K_{simpl}'$. The morphism $F: G(\cK_{simpl}') \ra G$ is injective on the local groups  as $G$ is a $C'(1/6)$--small cancellation group, see Theorem \ref{T: local factor injections}. The result thus follows from Proposition \ref{isompi1}.
\end{proof}

\begin{definition}[piece, $C'(1/6)$ polygonal complex] 
Let $Y$ be a polygonal complex. A \textit{path of $Y$} is an injective path in the $1$-skeleton of $Y$. For a path $P$ of $Y$, we denote by $|P|$ the number of edges of $P$, called its \textit{length}. 

A \textit{piece} of a polygonal complex $Y$ is a path $P$ of $Y$ such that there exist polygons $R_1$ and $R_2$ such that the map $P \ra Y$ factors as $P \ra R_1 \ra Y$ and $P \ra R_2 \ra Y$ but there does not exist a homeomorphism $\partial R_1 \ra \partial R_2$ making the following diagram commute: 
$$ \xymatrix{
     P  \ar[d]_{} \ar[r]_{}^{} & \partial R_2 \ar[d]^{} \\
    \partial R_1 \ar[r]_{} \ar[ur]& Y. \\
  }$$\\
By convention, we also consider edges of $Y$ as pieces.

The polygonal complex $Y$ is said to be a $C'(\lambda)$ \textit{polygonal complex}, $\lambda >0$, if for every piece $P$ of $Y$ and every polygon $R$ of $Y$ containing $P$ in its boundary, we have $|P| < \lambda \cdot |\partial R|$.
\label{smallcancellationcomplex}
\end{definition}

\begin{prop}
Let $G$ be a $C'(1/6)$--small cancellation group over the free product $\F$. Then, the polygonal complex $X$ defined above is a $C'(1/6)$ polygonal complex.
\end{prop}

\begin{proof}
Consider two polygons of $X$ sharing an edge. Up to the action of $G$, we can assume that such an edge contains the vertex $v_A$.  The two chosen polygons then correspond to two cyclic conjugates of $w$. By construction of $G(\cK_\mathrm{simpl}')$, these cyclic conjugates must be distinct. The result thus follows from the $C'(1/6)$--condition satisfied by $G$.
\end{proof}
The Greendlinger Lemma \cite{LyndonSchupp} immediately implies the following, see for instance  \cite[Lemma 13.2]{McCammondWiseFansLadders}. 
\begin{cor}
The polygons of $X$ are embedded.\qed
\end{cor}

\subsection{Complex of spaces with proper and cocompact action}

A group is \emph{cubulable} if it acts geometrically, i.e. properly discontinuously and cocompactly, on a CAT(0) cube complex.  From now on, we assume that $A$ and $B$ are cubulable groups, and denote CAT(0) cube complexes with a geometric action of $A$ and $B$ respectively by $EA$ and $EB$ respectively.

\medskip

Let $Y$ be a CW-complex. We consider the \emph{vertex set} of $Y$ as a metric space, equipped with the \emph{graph- or edge metric} on the $1$-skeleton of $Y$. We abuse notation and refer to this metric space again as $Y$.

\medskip

We now apply a useful theory for classifying spaces of complexes of groups \cite[Section~2]{MartinBoundaries}. This theory provides us with an explicit construction of a simply connected polyhedral complex with a geometric action of $G$. The construction can be thought of as of \emph{ blowing up}  the vertices of the polygonal complex $X$.

\begin{definition}[Definition 2.2 of \cite{MartinBoundaries}]\label{EZcomplexofspaces}
Let $G(\cY)$ be a complex of groups over a scwol $\cY$. A \textit{complex of classifying spaces $EG(\cY)$ compatible with the complex of groups $G(\cY)$} consists of the following:
\begin{itemize}
 \item For every vertex $\sigma$ of $\cY$, a space $EG_\sigma$, called a \textit{fibre}, which is a cocompact model for the classifying space for proper actions of the local group $G_\sigma$,
 \item For every edge $a$ of $\cY$ with initial vertex $i(a)$ and terminal vertex $t(a)$, a $G_{i(a)}$-equivariant map $\phi_a: EG_{i(a)} \ra EG_{t(a)}$, that is, for every $g \in G_{i(a)}$ and every $x \in EG_{i(a)}$, we have 
$$ \phi_{a}(g.x) = \psi_{a}(g).\phi_{a}(x),$$
and such that for every pair $(b,a)$ of composable edges of $\cY$, we have
$$g_{b,a} \circ \phi_{ba} = \phi_b  \phi_a.$$
\end{itemize}
\end{definition}

Complexes of classifying spaces compatible with a given complex of groups were shown to exist in full generality in \cite{MartinCombinationEG}. However, we define here an explicit complex of classifying spaces compatible with $G(\cK_\mathrm{simpl}')$.  We use this space  to define a wallspace structure in Section \ref{S: construction walls}. Let us denote by $c_i$ the barycentre of the radius of $\cK_{simpl}'$ from the apex to the vertex $v_i$. Recall from Definition \ref{complexofgroupsG} that $e_i$ is the edge of  $\cK_\mathrm{simpl}'$ starting at $s_i:=c_{2i}$ and terminating at  $v_{2i}$. Furthermore, recall that $f_i$ is the edge of  $\cK_\mathrm{simpl}'$ starting at $t_i:=c_{2i+1}$ and terminating at the vertex $v_{2i+1}$.

\begin{itemize}
\item The fibre $EG_{u_A}:=EA$ and $EG_{u_B}:=EB$ are the given CAT(0) cube complexes. We fix base vertices $x_A \in EG_{u_A}$ and $x_B \in EG_{u_B}$ respectively.

\item For each $i=0, 1,\ldots, N-1$, we choose an oriented   \emph{geodesic} 
\[
\gamma_{A,i} \hbox{  from $x_A$ to $a_i\cdot x_A$}
\]
in $EG_{u_A}$, and denote by $|\gamma_{A,i}|$ its edge length. 
Let $EG_{s_i}$ be the oriented simplicial segment of $|\gamma_{A,i}| $ edges, and let $\phi_{e_i}: EG_{s_i} \ra EG_{u_A}$ be a parametrisation of $\gamma_{A,i}$.

\item
 For each $i=0 \ldots, N-1$, we choose an oriented   \emph{geodesic}  
 \[
 \gamma_{B,i} \hbox{ from $x_B$ to $b_i\cdot x_B$}
   \]
 in $EG_{u_B}$, and denote by $|\gamma_{B,i}|$ its edge length. 
Let $EG_{t_i}$ be the oriented simplicial segment of $|\gamma_{B,i}| $ edges, and let $\phi_{f_i}: EG_{t_i}  \ra EG_{u_B}$ be a parametrisation of $\gamma_{B,i}$.

\item All the other fibres are reduced to a single point and all the other maps are the trivial ones.
\end{itemize}

It is straightforward to check that this indeed defines a complex of classifying spaces compatible with $G(\cK_\mathrm{simpl}')$.

\begin{definition}[The space $\cE G$]
We construct a space $\cE G$, obtained from the disjoint union
\[
\underset{0 \leq k \leq 2}{\coprod}~ \underset{\fra \in \frA^{(k)}(\cK_{simpl}')}{\coprod}
 \bigg( \lquotient{F_{ i(\fra)}(G_{i(\fra)})}{G} \times \{\fra\} \times \Delta^{k} \times EG_{i(\fra)} \bigg)
\]
by identifying pairs of the form 
\[
( [gF(a)^{-1}],\partial_i \fra, x, \phi_{a}(\xi) ) \hbox{ and } \big([g], \fra, d_i(x), \xi\big) 
\hbox{, for $0 \leq i \leq 
k$,}
\]
  where $a$ is the edge with initial vertex $i(\fra)$ and terminal vertex $i(\partial_i \fra)$. 
The various maps 
\[
\lquotient{F_{ i(\fra)}(G_{i(\fra)})}{G} \times \{\fra\} \times \Delta^{k} \times EG_{i(\fra)} \ra \lquotient{F_{ i(\fra)}(G_{i(\fra)})}{G} \times \{\fra\} \times \Delta^{k} 
\]
 obtained by forgetting the last coordinate yield a projection 
\[
 p: \cE G \ra X.
\]
The preimage of a vertex $v$ of $X$ under $p$  is called the \textit{fibre} over $v$ and denoted $EG_v$, as it is a cocompact model for the classifying space for proper actions of the stabiliser $G_v$ of $v$.
\end{definition}

\begin{figure}[H]
\begin{center}
\scalebox{0.7}{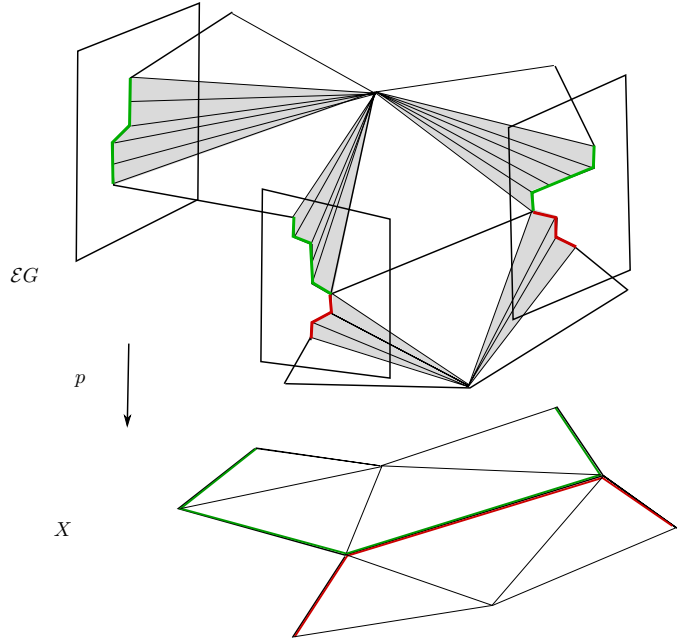}
\caption{The polyhedral structure of $\cE G$. The figure presents portions of two simplicial polygons of $X$ (one green, one red) and its preimage in $\cE G$. Vertical triangles are shaded. Attaching paths in the various fibres are coloured with respect to the associated polygon of $X$.}
\label{FigureEG}
\end{center}
\end{figure}

The following proposition is an application of Theorem 2.4 of \cite{MartinBoundaries}.
\begin{prop}
The space $\cE G$ is simply connected, and the $G$-action on it is proper and cocompact. \qed
\end{prop}

\begin{rmk}
This result was proven in \cite[Theorem 2.4]{MartinBoundaries} in the case of a complex of groups over a simplicial complex. Here, while $G(\cK_\mathrm{simpl}')$ is not a complex of groups over a simplicial complex, the geometric realisation of $\cK_\mathrm{simpl}'$ is nonetheless a simplicial complex, and the proof of \cite{MartinBoundaries} carries over to this case without any change.
\end{rmk}

\begin{definition}[polyhedral structure on $\cE G$, Figure \ref{FigureEG}]  The space $\cE G$ can be endowed with a polyhedral structure as follows. First note that we have, in particular, a projection $\cE G \to K_{simpl}'$. 
\begin{itemize}
\item Each fibre is isomorphic to a locally finite  CAT(0) cube complex, more specifically $EG_{v}$ is isomorphic to $EA$ if $v$ is a vertex in the preimage of $u_A$,  and $EG_{v}$ is isomorphic to $EB$ if $v$ is  a vertex in the preimage of $u_B$.

\item Let $R$ be a polygon of $X$, and denote by $\mathring{R}$ its interior. The boundary of $p^{-1}(\mathring{R})$ in $\cE G$ is a path of $\cE G$ which is the concatenation of   geodesics in the fibres (which are translates of the chosen geodesics $\gamma_{A,i}, \gamma_{B,i}$) and paths which map homeomorphically onto edges of $X_{\mathrm{simpl}}$. Thus, such a boundary comes equipped with a simplicial structure, and we identify the closure of $p^{-1}(\mathring{R})$ with the simplicial cone over such a boundary path. The preimage of the closure $p^{-1}(\mathring{R})$ with this simplicial structure is called a \textit{simplicial polygon} of $\cE G$. 
\end{itemize}
This endows $\cE G$ with a polyhedral structure, and the projection map $p: \cE G \ra X_{\mathrm{simpl}}$ is a polyhedral map.  
For a polygon $R$ of $X$, we denote by 
\[
 \widetilde{R}:= \hbox{closure of $p^{-1}(\mathring{R})$}
\]
 the associated simplicial polygon of $\cE G$.
\end{definition}

\begin{definition}[horizontal, vertical polyhedrons]\label{vertical and horizontal}
We say that a polyhedron of $\cE G$ is \textit{horizontal} if $p$ restricts to a homeomorphism on it, and \textit{vertical} otherwise. For an edge $e$ of $X$, we denote by $\widetilde{e}$ the unique horizontal edge of $\cE G$ which maps onto $e$ under $p$.
\end{definition}

\begin{definition}[attaching paths]
Let $R$ be a polygon of $X$ and $v$ be a vertex of $R$. We define the \textit{attaching path of $\widetilde{R}$ along $EG_v$}: 
$$p_{v,R} := EG_v \cap \widetilde{R}.$$
\label{attachingpath}
\end{definition}

In this section we have constructed a polygonal complex $\cE G$ which is a realisation of an analogue for quotients of free products of the  Cayley complex of the group $G$ . In particular, the group $G$ acts properly and cocompactly on $\cE G$. The space $\cE G$ was realised as a complex of spaces over the $C'(1/6)$--small cancellation polygonal complex $X$ constructed in Section \ref{complexofgroupsfreeproduct}. $\cE G$ is equipped with a polyhedral structure consisting of the following two \emph{building blocks}:   \emph{polygons} of $\cE G$, which are mapped to polygons of $X$ (the latter being modelled after the polygon $\widetilde{R_0}$), and \emph{CAT(0) cube complexes}, which are fibre of vertices of $X$, and are isomorphic to the chosen complexes $E_A$ or $E_B$. 

Our next aim is to describe a wall structure on $\cE G$. One family of walls on $\cE G$ is obtained by lifting the walls of $X$. A second family of walls is obtained by combining  natural wall structures  on the  polygons of $\cE G$ and the various CAT(0) cube complexes. There is however, a priori, no canonical way to combine these walls, see our explanation in Section \ref{S: construction walls}. The geometric structure of the corresponding wallspace associated with $\cE G$ is controlled using the properties of the $C'(1/6)$--small cancellation polygonal complex $X$ in combination with the properties of the fibre CAT(0) cube complexes. The properties of $X$ are discussed in Section \ref{S: galleries} and the Appendix.

\section{The wallspace}

Spaces with walls were introduced by Haglund--Paulin \cite{HaglundPaulinWallSpaces} 
and generalise essential properties of CAT(0) cube complexes. 
\begin{definition}[wallspace]
A \textit{wallspace} is a pair ($Y, \cH$) consisting of a set $Y$ together with a collection $\cH$ of non-empty subsets of~$Y$, called \textit{half-spaces}, such that: 
\begin{itemize}
\item for every half-space $H$ in $\cH$, its complement $Y \setminus H$ is also in $\cH$,
\item for every $x,y$ of $Y$, there are only finitely many half-spaces $H$ such that $x\in H$ and $y \notin H$.
\end{itemize}
A partition of $Y$ into two half-spaces is called a \textit{wall}, and we denote the set of walls of ($Y, \cH$) (short: $Y$) by $\cW(Y)$.  \label{wallspace}
\end{definition}
We say that a wall \textit{separates} a pair of points of $Y$ if each half-space associated with that wall contains exactly one point of the pair. We say that two walls $W=\{H, Y \setminus H\}$ and $W'= \{H', Y \setminus H'\}$ \textit{cross} if all the intersections $H \cap H',  H \cap (Y \setminus H'), (Y \setminus H) \cap H', (Y \setminus H) \cap (Y \setminus H')$ are non-empty. We define the \textit{wall-pseudometric} $d_{\cW(Y)}(x,y)$ between two points $x,y$ of $Y$ to be the number of walls separating them. We say that a group \textit{acts on a wallspace} if it acts on the underlying set and preserves the set of half-spaces.

\begin{definition}[wallspace on a polyhedral complex] A structure of wallspace \textit{on a polyhedral complex} is a structure of wallspace on its vertex set. 
\end{definition}

If a wall of a polyhedral complex is defined by means of the complement of a separating subset containing no vertex, we will \emph{abuse notation} and not distinguish the associated wall and the separating subset.

\medskip

Whenever a group acts on a space with walls, one can associate an action of the group on a CAT(0) cube complex by isomorphisms. The CAT(0) cube complex can explicitly be described using the walls, see \cite{ChatterjiNibloWallSpaces,nica_cubulating_2004} for the explicit  construction. Let $G$ be a small cancellation group over the free product of two groups. 
The aim of this section is to define a set of walls $\cW$ on the polyhedral complex $\cE G$, turning $\cE G$ into a wallspace. The above mentioned general procedure then yields  the cube complex $C_\cW$ associated with the action of $G$ on the wallspace $(\cE G, \cW)$.

Again,  $G$ denotes the small cancellation quotient $\rquotient{A*B}{\ll w \gg}$, and $X$ the $C'(1/6)$--polygonal complex constructed in Section \ref{SectionSmallCancellation}. 
 
\subsection{Galleries, hypercarriers and hypergraphs}\label{S: galleries}

In this section we introduce fundamental notions and theory that we use later to define walls and then to study their  geometric structure. In what follows, while results are stated for the polygonal complex $X$, the results hold for an arbitrary $C'(1/6)$--polygonal complex.

\begin{definition}[far apart]\label{DefinitionFarApart}
Let $R$ be a polygon of a $C'(1/6)$--polygonal complex and $\tau_1$, $\tau_2$ two simplices of its boundary $\partial R$. We say that $\tau_1$ and $\tau_2$ are \textit{far apart} in $R$ 
if no path $P$ in $\partial R$ containing both $\tau_1$ and $\tau_2$ is a concatenation of strictly less than four pieces.
\end{definition}

\begin{example}
In a $C'(1/6)$ polygonal complex, opposite edges of a given polygon are far apart.
\end{example}

If two cells of a given polygon $R$ of a $C'(1/6)$ polygonal complex are far apart in $R$, then the polygon $R$ is unique by the small cancellation condition. We thus simply say that these cells are \textit{far apart}, the reference to $R$ being implicit.

\begin{definition}[polygon with doors, system of doors]\label{D: doors}
A \textit{polygon with doors} is a polygon $R$ of $X$, referred as the \textit{underlying cell}, together with a choice of simplices $\tau_1, \tau_2$ of $ \partial R$ called \textit{doors}. We will denote such a data $R_{\{\tau_1, \tau_2\}}$. (We  often write $R_{\{\tau_1, \tau_2\}}$ indistinctly for a polygon with doors and for its underlying cell.)

A \textit{system of doors} is a collection $\cC$ of polygons with doors. We will simply speak of a \textit{polygon} of $\cC$ when speaking of a polygon with doors of $\cC$. A \textit{door} of $\cC$ is a door of a polygon of $\cC$.  
\end{definition}

Note that a door can be an edge as well as a vertex in the boundary of a polygon.

\begin{definition}[Gallery]\label{D: gallery}
A \textit{gallery} is a system of doors $\cC$ satisfying the following conditions.
\begin{itemize}

\item (coherence condition) For every pair of polygons $R_{\{\tau_1, \tau_2\}}, R_{\{\tau_1', \tau_2'\}}$ of $\cC$ with the same underlying cell and such that $\tau_1= \tau_1'$, we also have $\tau_2=\tau_2'$.
\item (far apart condition) For every polygon $R_{\{\tau_1, \tau_2\}}$ of $\cC$, the doors $\tau_1$ and $\tau_2$ are far apart in the sense of Definition \ref{DefinitionFarApart}.
\item (connectedness condition) For every pair of doors $\tau, \tau'$ of $\cC$, there exists a sequence 
\[R_{\{\tau_1, \tau_2\}}, R_{\{\tau_2, \tau_3\}}, \ldots, R_{\{\tau_{n-1}, \tau_n\}}\]
of polygons of $\cC$ such that $\tau = \tau_1$ and $\tau'= \tau_n$.
\end{itemize}
\end{definition}

\begin{definition}[hypercarrier and hypergraph associated with a gallery]
Given a gallery $\cC$, we associate a polygonal complex to it as follows.  Take the disjoint union of all polygons $R_{\{\tau_1,\tau_2\}}$ of $\cC$.
Whenever $P$ is a path embedded in  $\partial R_{\{\tau_1,\tau_2\}}$ and $\partial R_{\{\tau_2,\tau_3\}}$, and if $P$ embeds in $X$ such that $P$ is contained in  the intersection of $\partial R_{\{\tau_1,\tau_2\}}$ and $\partial R_{\{\tau_2,\tau_3\}}$ in $X$, then we identify  $\partial R_{\{\tau_1,\tau_2\}}$ and $\partial R_{\{\tau_2,\tau_3\}}$ along $P$. 
 The resulting polygonal complex is denoted by $Y_\cC$ and called the \textit{hypercarrier} associated with $\cC$.

For each polygon $R_{\{\tau_1,\tau_2\}}$ of $\cC$, we denote by $L_{\{\tau_1,\tau_2\}}$ the path of $R_{\{\tau_1,\tau_2\}}$ which is the union of the radii of  $R_{\{\tau_1,\tau_2\}}$ joining the apex of $R_{\{\tau_1,\tau_2\}}$ to the barycentres of $\tau_1$ and $\tau_2$. Let
$$ \Lambda_\cC := \bigcup L_{\{\tau_1,\tau_2\}} \subset Y_\cC.$$
We call $\Lambda_\cC$ the \textit{hypergraph} associated with $\cC$.
\end{definition}
The hypercarrier $Y_\cC$ comes endowed with a map $i_\cC:Y_\cC \ra X$, by mapping every polygon in $Y_\cC$ to the corresponding polygon in $X$. This map is by construction an immersion on the $1$-skeletons.

We note that our hypercarriers and hypergraphs extend the corresponding notions of Wise \cite[Definition 3.2 and 3.3]{WiseSmallCancellation}. In particular, Wise's hypercarriers and hypergraphs are defined by means of opposite edges, see Section  \ref{ExampleEdge}. Our far apart condition allows, in contrast, the study of  hypergraphs and hypercarriers that are \emph{not} associated with opposite edges. Our definition moreover includes hypergraphs going through the vertices of $X$. The hypercarriers we consider are therefore allowed to have cutpoints at such vertices,~cf.~Figure~\ref{F: galleries, hypergraphs}. Such configurations do not appear in \cite{WiseSmallCancellation}. 

\begin{figure}[H]
\begin{center}
\scalebox{1}{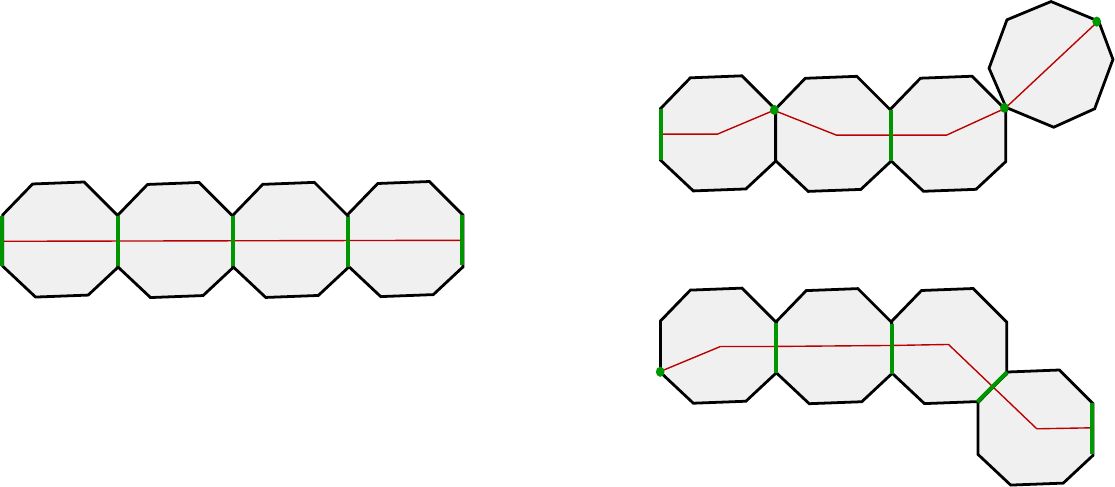}
\caption{Examples of hypercarriers with their associated hypergraphs (in red) and doors (in green). Configurations on the left are studied in detail in  \cite{WiseSmallCancellation}.  
The configurations on the right are studied in detail in the appendix.
}
 \label{F: galleries, hypergraphs}
\end{center}
\end{figure}

\begin{definition}[ convex]
A subcomplex $Y$ is called \emph{ convex} if every   geodesic between two vertices of $Y$ is contained in $Y$.
\end{definition}

The following results extend Lemma 3.11 and Theorem 3.18 of \cite{WiseSmallCancellation}, cf. Proposition \ref{P: wise hyper}.

\begin{thm}[cf. Proposition \ref{simplyconnectedembedded}, Corollary \ref{A:hypergraphtree} and Proposition \ref{A: lyconvex}]\label{T: properties hyper}
Let $\cC$ be a gallery in $X$. Then:
\begin{itemize}
 \item Its hypercarrier $Y_\cC$ is connected and simply connected and the map $i_\cC:Y(\cC) \ra X$ is an embedding.
 \item  The associated hypergraph $\Lambda_\cC$ is a tree which embeds in $X$.
 \item  The subcomplex $Y_\cC$ of $X$ is  convex.
\end{itemize}
\end{thm}
\begin{cor}\label{C: polygonsareconvex} Polygons of $X$ are  convex. \qed
\end{cor}

The proofs are by - now standard - small cancellation arguments, and extend the original arguments of Wise in a straightforward way, using our far apart condition. We give a complete account of the arguments in~Appendix~\ref{Appendix A}. 

We study several examples of galleries, hypergraphs and hypercarriers below. In Section \ref{ExampleEdge} we review Wise's hypergraphs and hypercarriers associated with diametrically opposed edges in the $C'(1/6)$--small cancellation complex $X$. In Section \ref{S: walls coming from X} we lift such hypergraphs and hypercarriers to $\cE G$. Finally, in Section \ref{S: extending the hyperplanes of fibres} we modify $\cE G$ to extend the hyperplanes in the fibres of $\cE G$. We obtain graphs of spaces whose projection to $X$ are hypergraphs associated with a gallery of $X$. In all three situations, we show that the complement of a hypergraph defines a wall. Here Theorem \ref{T: properties hyper} is essential. 

\subsection{Walls on the building blocks}

Recall that the space $\cE G$ has two building blocks, the polygons of $\cE G$ and the CAT(0) cube complexes which are fibres of vertices of $X$. Its geometric structure and the combination of these building blocks is controlled using the   properties of the underlying $C'(1/6)$--small cancellation polygonal complex $X$. For these three types of spaces,  the  $C'(1/6)$--small cancellation polygonal complex, the  polygons of $\cE G$, and the fibre CAT(0) cube complexes, we describe the associated wallspace structures.

\subsubsection{Walls of diametrically opposed edges}\label{ExampleEdge}

As usual, $X$ denotes the $C'(1/6)$--polygonal complex constructed in Section \ref{complexofgroupsfreeproduct}. Note that what follows can be applied to an arbitrary $C'(1/6)$--polygonal complex. 

We first put a wall structure on $X$.  Then we discuss hypergraphs and walls on polygons of $\cE G$.
We define an equivalence class on the set of edges of $X$ as follows. Two edges $e$ and $e'$ are said to be \textit{diametrically opposed } or \emph{opposite} if there exists a polygon $R$ containing them and such that $e$ and $e'$ are diametrically opposed in $R$. 
 We denote by $R_{\{e,e'\}}$ the associated polygon with doors. 
\begin{definition}[Equivalence class of  opposite edges]
Two edges $e$ and $e'$ are \textit{equivalent} if there is a sequence $e=e_1, \ldots, e_n=e'$ of edges such that any two consecutive ones are diametrically opposed. 
\end{definition}
For an edge $e$ of $X$, we define the complex with doors $\cC_e^{X}$ to be the disjoint union of all the polygons with doors $R_{\{e_1,e_2\}}$ where $e_1, e_2$ are diametrically opposed and in the equivalence class of $e$. Observe that $\cC_e^X$ is a gallery by definition. The far apart condition follows immediately from the fact that $X$ is a $C'(1/6)$--polygonal complex.   We denote the  associated hypergraph by $\Lambda_e^{X}$,  and the associated hypercarrier by $Y_e^{X}$. This coincides with  Wise's hypergraphs and hypercarriers \cite[Definition 3.2, 3.3]{WiseSmallCancellation}. 
Theorem \ref{T: properties hyper} implies:
\begin{prop}[{\cite[Lemma 3.11, Theorem 3.18]{WiseSmallCancellation}}]\label{P: wise hyper}
Every hypergraph $\Lambda_e^{X}$ embeds in $X$, is contractible and separates $X$ into two connected components. \qed
\end{prop}

\begin{definition}[Walls on $X$] For every edge $e$ of $X$, the associated hypergraph $\Lambda_e^{X}$ separates $X$ in two components. Let $W_e^X$ be the wall of $X$ associated with this decomposition. We say that $W_e^X$ is the wall \emph{associated with} $e$.

Let $\cW^X$ be the set of all these walls.
\end{definition}

\begin{prop}[\cite{WiseSmallCancellation}]
 The space $X$ with the walls $\cW^X$ is a wallspace. The wall pseudometric on $X$ is a metric. \qed
\end{prop}

In Section \ref{S: walls coming from X}, we lift the walls $\cW^X$ to $\cE G$.

 \begin{rmk}[Hypergraphs and Walls on polygons]\label{R: walls polygon}
  Consider a  single polygonal cell $R$ on an even number of edges as a $C'(1/6)$--small cancellation polygonal complex. It then comes  with the above defined hypergraphs and walls of diametrically opposed edges. We denote the hypergraph of $R$ associated with $e$ by $\Lambda^{R}_e$. The corresponding wall on $R$ is denoted by $W_e^{R}$.  
\end{rmk}

Up to taking  a subdivision of $\cE G$, this, in particular, endows each polygon of $\cE G$ with a wallspace structure.

\subsubsection{Hyperplanes in CAT(0) cube complexes}

 We recall some facts on hyperplanes in  CAT(0) cube complexes. Let $C$ be a CAT(0) cube complex. The building blocks of $C$ are cubes, each $k$-cubing isomorphic to $[-1,1]^{k}$ for some integer $k \geq  0$. A cube hyperplane associated with a cube $I$ is obtained by setting exactly one coordinate to zero, and is therefore of the form $[-1,1]^{i}\times \{0\}\times [-1,1]^{j}$ with $i+j=k-1$. A hyperplane on $C$ is a connected nonempty subspace whose intersection with each cube $I$ of $C$ is either empty or a cube hyperplane associated with $I$. Every edge of $C$ has  a unique hyperplane intersecting it.

\begin{prop}\label{P: walls cupe complexes} \cite[Th. 4.10, Th. 4.13]{SageevCubeComplex}
 Let $H$ be a hyperplane of $C$.
\begin{itemize}
\item The hyperplane $H$ is contractible and separates $C$ into two connected components. 
 \item The   neighbourhood of a hyperplane $H$ is  convex. \qed
\end{itemize}
\label{propertieshyperplane}
\end{prop}

In particular, given two vertices of $C$ there is a hyperplane separating them, and every hyperplane defines a wall of $C$. The following follows from the work of Sageev  \cite{SageevCubeComplex}.

\begin{prop}\label{P: cube complex is wallspace}
 A CAT(0) cube complex $C$ with the collection of the complements of its hyperplanes as walls is a wallspace. The wall pseudometric on $C$ is a metric. \qed
\end{prop}

In Section \ref{S: extending the hyperplanes of fibres} we extend the walls in the fibres $EG_v$, using the walls on the polygons of $\cE G$. We therefore need the following observations.
\begin{lem}\label{L: cube complexes} Let $v$ be a vertex in $X$, and let $EG_v$ be the corresponding fibre in $\cE G$. Let $p_{v,R}$ be the attaching path where $R$ is a polygon $R$ of $X$. Suppose $H$ is a hyperplane that crosses an edge $e$ of $p_{v,R}$. Then, 
 \begin{itemize}
\item (Fibre separation) The hyperplane $H$ separates the vertices of $e$ \emph{in $EG_v$}.  In particular, the hyperplane intersects every path in $EG_v$ that connects the starting and endpoint of the attaching path $p_{v,R}$.
 \item  (No turns) The hyperplane $H$ does not intersect $p_{v,R}$ more than once.\qed
\end{itemize}\end{lem}
 The first fact is immediate from the above properties of CAT(0) cube complexes. For the second fact recall that $p_{v,R}$ is geodesic in $EG_v$. Hence, a turn would contradict the   convexity of hyperplanes in CAT(0) cube complexes.

\subsection{Construction of the new walls}\label{S: construction walls}

In this section we lift the walls of $X$ to $\cE G$, and explain how to combine the walls on the building blocks of $\cE G$. 
The space $\cE G$ is build up from the various CAT(0) cube complexes $EG_v$, modelled after the CAT(0) cube complexes $E_A$ and $E_B$, and the various polygons of $\cE G$. We just saw that these building blocks of $\cE G$ are equipped with natural wallspace structures. The idea is to combine walls defined by the hyperplanes on the fibre CAT(0) cube complexes with the walls of opposite edges for polygons of $\cE G$. We now observe  that there is a priori no canonical way to do this. In particular, it is  not possible to  employ the viewpoint of Wise's seminal paper \cite[Section~5]{wise_structure_2011}: To adapt to the viewpoint of Wise, view the boundary path of a polygon $\widetilde{R}$ of $\cE G$ as a cube complex, and $\widetilde{R}$ as a cone over this boundary path. It comes with the wall structure 
associated with opposite edges. Combining  the walls of $E_A$, $E_B$ and $\widetilde{R}$  as in  \cite[Section~5.f]{wise_structure_2011}, cf. Definition \ref{D: extended walls} below, does not yield walls;  in 
particular, conditions (1), (2) and (3) of Lemma 5.13 in \cite{wise_structure_2011} fail. Indeed, the subspaces we obtain
with such a procedure no longer embed. More precisely, as the small cancellation condition over the free product of two groups does not control the length of the attaching paths,  a hypergraph of diametrically opposed edges of $\widetilde{R}$ is likely to  intersect two distinct edges of the same attaching path of the  same fibre. The corresponding new hyperplane then consists of the two distinct hyperplanes associated with the aforementioned edges of that fibre and the hypergraph of diametrically opposed edges intersecting them. Note that we have no control of the   position of these two hyperplanes of the fibre cube complex, meaning that they can intersect, osculate, or just not intersect any other attaching path, hence the claim.

\subsubsection{Balancing}

We now modify the complex $\cE G$. This then allows us to combine the hyperplanes in the various CAT(0) cube complexes with the walls associated with opposed edges in polygons~of~$\cE G$. 

\begin{definition}[the subdivided complexes $X_k$ and $(\cE G)_k$] \label{balancing}
Let $k \geq 0$ be an even integer. We define a new polygonal structure from $X$ by subdividing each edge of $X$ exactly $k$ times. We denote by $X_k$ the resulting polygonal complex.

Similarly, we define a new polyhedral structure from $\cE G$ by subdividing each horizontal edge, see Definition \ref{vertical and horizontal}, exactly $k$ times. 
 We denote by $(\cE G)_k$ this new polyhedral structure, and~by 
\[
p: (\cE G)_k \ra X_k
\]
  the induced projection map.
\end{definition}

Note that this procedure does \textit{not} modify the CAT(0) cubical structures of the various fibres of $\cE G$, and it does \emph{not} modify the attaching paths.  Moreover, each complex $X_k$ does again satisfy the $C'(1/6)$--condition, and pieces of $X_k$ are subdivisions of pieces of $X$.

\begin{definition}(balanced)\label{balanced} We say that $(\cE G)_k$ is \textit{balanced} if for every polygon $\widetilde{R}$ of $(\cE G)_k$ and every edge $e$  of $\widetilde{R}$ with  opposite edge $e'$, the projections $p(e)$ and $p(e')$ are far apart (see Definition \ref{DefinitionFarApart}) in $X_k$. 
\end{definition}

\begin{lem}
There exists an even integer $k\geq 0$ such that $(\cE G)_k$ is balanced.
\end{lem}

\begin{proof} 
Since the number of edges in the various attaching paths $p_{v,R}$ is uniformly bounded above by the maximum of the edge lengths of the   geodesics $\gamma_{A,i}, \gamma_{B,i}$, the subdivided complex $(\cE G)_k$ becomes balanced for $k$ large enough by the $C'(1/6)$--condition.
\end{proof}

\begin{definition}
Let $k\geq 0$ be the smallest even number such that $(\cE G)_k$ is balanced. We denote by $\cE G_\mathrm{\textit{bal}}$ and $X_\mathrm{\textit{bal}}$ the complexes $(\cE G)_k$ and $X_k$ respectively.
\end{definition}

In the next section \emph{the  properties  of $X$}, in combination with the properties  of the fibre CAT(0) cube complexes, will be used to control the  geometric structure of $\cE G$. We first endow $\cE G$ with a wall structure.

\subsubsection{Lifted hypergraphs}\label{S: walls coming from X}

The polygonal complexes $X$ and $X_\mathrm{\textit{bal}}$ satisfy the small cancellation condition $C'(1/6)$, hence the hypergraphs of diametrically opposed edges  of Section \ref{ExampleEdge} define a wallspace on $X_{bal}$ 
We now lift  the corresponding family of walls on $X_{bal}$ to define a first family of walls on $\cE G_{bal}$. 

\begin{definition}[hypergraph associated with an edge of $X_\mathrm{\textit{bal}}$]
Let $e$ be an edge of $X_\mathrm{\textit{bal}}$ and $\Lambda_e^{X}$ the hypergraph of diametrically opposed edges in $X_\mathrm{\textit{bal}}$ defined in Section \ref{ExampleEdge}. We call $\Lambda_e^{X}$ the hypergraph \textit{associated with the edge} $e$ of $X_\mathrm{\textit{bal}}$.  

We define the subset $\widetilde{\Lambda_e^X}$ of $\cE G_{bal}$ as the preimage of $\Lambda_e^X$ under $p:\cE G_{bal} \ra X_{bal}$.
  We call $\widetilde{\Lambda_e^X}$ the \textit{lifted hypergraph} (of $\cE G_\mathrm{\textit{bal}}$) \textit{associated with the edge} $e$ of $X_\mathrm{\textit{bal}}$.

\end{definition}

\begin{lem}\label{WallXseparates}
Each lifted hypergraph $\widetilde{\Lambda_e^X}$ of $\cE G$ associated with an edge of $X_\mathrm{\textit{bal}}$ is contractible and separates $\cE G_\mathrm{\textit{bal}}$ into two connected components.
\end{lem}

\begin{proof}
We use Proposition \ref{P: wise hyper}. 
  Note that $p$ restricts to a homeomorphism $\widetilde{\Lambda_e^X}\ra \Lambda_e^{X}$. 
Hence, $\widetilde{\Lambda_e^X}$ is contractible. The fact that $\widetilde{\Lambda_e^X}$ disconnects $\cE G_\mathrm{\textit{bal}}$ follows from the fact that $\Lambda_e^{X}$ disconnects $X_\mathrm{\textit{bal}}$ into two components. The fact that $\cE G_\mathrm{\textit{bal}} - \widetilde{\Lambda_e^X}$ has exactly two connected components follows from the fact that the preimage of a connected set under $p$ is again connected.
\end{proof}

\begin{definition}[wall of $\cE G_{bal}$ associated with an edge of $X_\mathrm{\textit{bal}}$]\label{D: walls 1}
 We define the \textit{wall of} $\cE G_\mathrm{\textit{bal}}$ \textit{associated with the edge} $e$ of $X_\mathrm{\textit{bal}}$ as $W_e^X := \widetilde{\Lambda_e^X}$.
\end{definition}

Note that this family of walls is not large enough to define a wallspace structure on $\cE G$ whose associated CAT(0) cube complex is endowed with a proper action, as this family of walls does not separate vertices in a given fibre.

\subsubsection{Combining the walls on the building blocks}\label{S: extending the hyperplanes of fibres}

In this section,  we combine  walls on the building blocks of $\cE G$  to a wall on the whole space $\cE G_\mathrm{\textit{bal}}$.  
Let $e$ be an edge of $\cE G_\mathrm{\textit{bal}}$. 
If $e$ is a vertical edge (that is, contained in one of the fibre CAT(0) cube complexes), we denote by $H_e$ the hyperplane in that fibre associated with $e$.
 If $e$ is a horizontal edge (that is, projects to an edge of $X_\mathrm{\textit{bal}}$), we denote by $H_e$ the midpoint of $e$.  In both cases we call $H_e$ the \textit{hyperplane associated with} $e$.

\begin{definition}\label{D: extended walls}
We define an elementary equivalence relation on the set of edges of $\cE G_\mathrm{\textit{bal}}$ as follows. Two edges $e, e'$ of $\cE G_\mathrm{\textit{bal}}$ are said to be \textit{elementarily equivalent}, and we denote it $e \sim_1 e'$, if one of the following situations occurs:
\begin{itemize}
\item $e$ and $e'$ are  opposite edges in some polygon of $\cE G_\mathrm{\textit{bal}}$,
\item $e$, $e'$ are vertical edges in the same fibre and the hyperplanes $H_e$ and $H_e'$ coincide.
\end{itemize}
The transitive closure defines an equivalence relation  on the set of edges of $\cE G_\mathrm{\textit{bal}}$.
\label{equivalencerelationedgebal}
\end{definition}

\begin{definition}[systems of doors associated with an edge of $\cE G_\mathrm{\textit{bal}}$]
Let $e$ be an edge of $\cE G_\mathrm{\textit{bal}}$. We associate to $e$ a system of doors $\cC_e^{\cE G}$ of $X$ as follows. To every polygon $\widetilde{R}$ of $\cE G_\mathrm{\textit{bal}}$ together with a pair of diametrically opposed edges $e_1, e_2\in \widetilde{R}$ in the equivalence class of $e$, we associate a polygon with doors of $\cC_e^{\cE G}$ with underlying cell $p(\widetilde{R})$ and with doors being the projections $p(e_1)$ and $p(e_2)$.
\end{definition}

\begin{prop}\label{P: extended wall gallery}
The system of doors $\cC_e^{\cE G}$ is a gallery.\qed
\end{prop}
\begin{proof} We have to verify the conditions listed in Definition \ref{D: gallery}. The connectedness condition follows immediately from the definition. The doors of a given polygon of $\cC_e^{\cE G}$ are far apart because $\cE G_\mathrm{\textit{bal}}$ is balanced. 
Suppose by contradiction that there exists a pair of polygons of $\cC_e^{\cE G}$ violating the coherence condition, i.e. a pair of polygons $R_{\{\tau_1, \tau_2\}}, R_{\{\tau_1', \tau_2'\}}$ of $\cC_e^{\cE G}$ with the same underlying cell $R$, such that $\tau_1= \tau_1'$ and $\tau_2\neq\tau_2'$. By the connectedness condition, let \[R_{\{\tau_1, \tau_2\}}, R_{\{\tau_2, \tau_3\}}, \ldots, R_{\{\tau_{n-1}, \tau_n\}}\] be a sequence of polygons of $\cC_e^{\cE G}$ with $\tau_{n-1}=\tau_2', \tau_n=\tau_1$, and \[R_1, R_2, \ldots, R_{n-1}\] the associated sequence of underlying cells, with $R_1=R_{n-1}=R$. We can assume that such a sequence of polygons is minimal.

We claim that, for $1 < i < j \leq n-1$, we have $R_i \neq R_j$. Indeed, if this was not the case, then either the set of doors $\{\tau_i, \tau_{i+1}\}$ of $R_i$ is disjoint from the set of doors $\{\tau_j, \tau_{j+1}\}$ of $R_j$, or those polygons share a door. In the former case, $R_{\{\tau_i, \tau_{i+1}\}}, \ldots, R_{\{\tau_{j}, \tau_{j+1}\} }$ defines a gallery (the coherence condition now being trivially verified), the hypercarrier of which does not embed, contradicting Theorem \ref{T: properties hyper}. In the latter case, this contradicts the minimality of the initial sequence of polygons. This proves our claim.

It now follows that $R_{\{\tau_2, \tau_{3}\}}, \ldots, R_{\{\tau_{n-1}, \tau_{n}\} }$ defines a gallery (the coherence condition being trivially verified), and we have $R_2 \cap R_{n-1} \neq \varnothing$ by hypothesis. Since $n \geq 4$ by the far apart condition, it follows that the hypercarrier of that gallery does not embed in $X$, contradicting Theorem~\ref{T: properties hyper}.
\end{proof}

Let $\Lambda_e^{\cE G}$ be the hypergraph in $X$ associated with $\cC_e^{\cE G}$. 
It follows from Proposition \ref{P: extended wall gallery} and Proposition \ref{T: properties hyper} that $\Lambda_e^{\cE G}$ is a tree.
\begin{lem}\label{L: hypergraph tree}
 The hypergraph $\Lambda_e^{\cE G}$ is a tree embedded in $X$.\qed
\end{lem}

\begin{definition}[wall of $\cE G$ associated with an edge of $\cE G_\mathrm{\textit{bal}}$]
Let $e$ be an edge of $\cE G_\mathrm{\textit{bal}}$. We define the \textit{wall associated with} $e$ as a tree of spaces over the hypergraph $\Lambda_e^{\cE G}$ as follows.
Let $\widetilde{R}$ be a polygon of $\cE G$ and let $e_1$ and $e_2$ be  opposite edges of a polygon $\widetilde{R}$ of $\cE G_\mathrm{\textit{bal}}$, which are in the equivalence class of $e$, see Definition \ref{equivalencerelationedgebal}. Note that the polygon $p(\widetilde{R})$  of $X_\mathrm{\textit{bal}}$, together with the doors $p(e_1)$ and $p(e_2)$, defines a polygon of $\cC_e^{\cE G}$. We define  
$$ W_{e}^{\cE G} := \underset{\mathrm{in~the~equivalence~class~of~} e}{\bigcup_{e_1 \sim_1 e_2 }} \big( H_{e_1} \cup  \Lambda_{e_1}^{\widetilde{R}} \cup H_{e_2} \big),$$
where $\Lambda_{e_1}^{\widetilde{R}}$ is the hypergraph of the polygon $\widetilde{R}$ defined in Remark \ref{R: walls polygon}. $W_{e}^{\cE G}$ is called the \textit{wall of $\cE G$ associated with $e$}.
\label{D: walls 2}
\end{definition}

We readily observe that the above defined wall $W^{\cE G}_e$ is a combination of hyperplanes of the various CAT(0) cube complexes of $\cE G$ and hypergraphs  of the various polygons of $\cE G$.

Note that the projection of the wall $W_e^{\cE G}$ under $p:\cE G_{bal}\to X_{bal}$ is the hypergraph $\Lambda_e^{\cE G}$ associated with the gallery $C_e^{\cE G}$. 
Let us distinguish two types of walls $W_e^{\cE G}$ associated with an edge of $\cE G_\mathrm{\textit{bal}}$ according to their projections $\Lambda_e^{\cE G}$ in $X$. 
 \begin{itemize}
  \item The wall $W_e^{\cE G}$ and its associated hypergraph $\Lambda_e^{\cE G}$ are said to be \emph{of  first type} if $\Lambda_e^{\cE G}$ consists of a single vertex.
  \item Otherwise, $W_e^{\cE G}$ and  $\Lambda_e^{\cE G}$ are said to be \emph{of second type}. 
 \end{itemize}
 
 \begin{figure}[H]
\begin{center}
\scalebox{0.4}{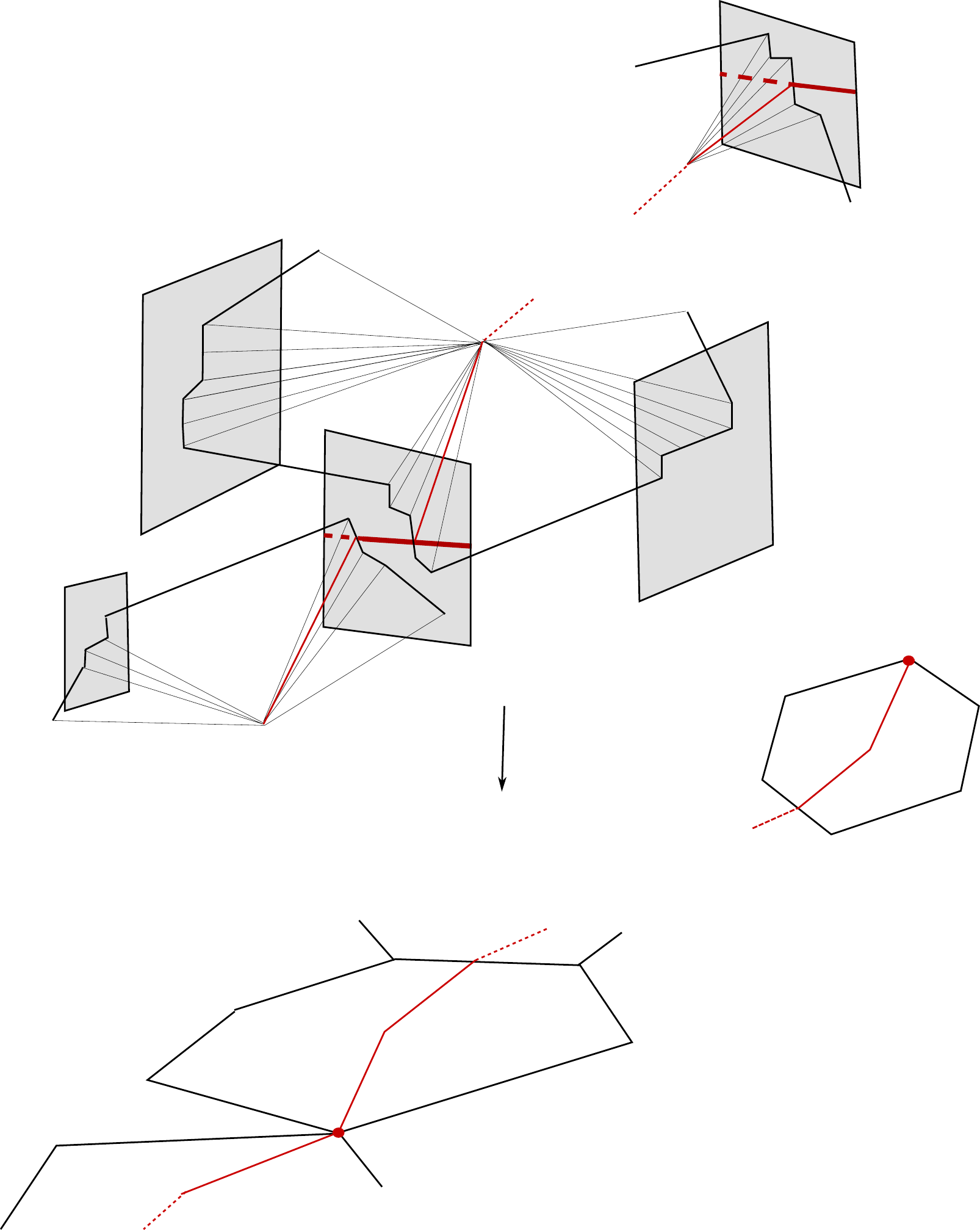}
\caption{A portion of a wall associated with an edge of $\cE G_\mathrm{\textit{bal}}$ of second type, together with its hypergraph. To avoid drawing too many edges, we assume here that $\cE G$ is balanced, that is, $\cE G= \cE G_{\mathrm{\textit{bal}}}$.}
\label{FigureWall}
\end{center}
\end{figure}
 
Note that a wall $W_e^{\cE G}$ associated with a vertical edge $e$ of $\cE G_{\mathrm{\textit{bal}}}$ is of first type if and only if $e$ is contained in a fibre CAT(0) cube complex and the associated hyperplane crosses none of the attaching paths defined in Definition \ref{attachingpath}.
An example where all occurring types of hypergraphs $\Lambda_e^{\cE G}$and  $\widetilde{\Lambda_e^{\cE G}}$ are displayed is shown in Figure \ref{F: 3 types of walls}.
 
 We now show that the walls of $\cE G$ associated with edges of $\cE G_{\mathrm{\textit{bal}}}$ are walls in the sense of Definition \ref{wallspace}, that is, they separate $\cE G$ into exactly two connected components. As noted in the introduction, the results and methods of \cite[Section~5]{wise_structure_2011} cannot be applied to conclude in our situation. Instead, we use, as already mentioned, the   properties of hypercarriers in the $C'(1/6)$--polygonal complex $X$ and the properties of the fibre CAT(0) cube complexes. Hence, we give a more direct approach to the cubulation problem.
 
  \begin{lem}\label{WallEGseparates1}
A wall associated with an edge of $\cE G_\mathrm{\textit{bal}}$ of first type is contractible and separates $\cE G$ into two connected components.
\end{lem}
 \begin{proof}
 Let $e$ be an edge of $\cE G_{bal}$ whose associated hypergraph is of first type. The hypergraph is then completely contained in a CAT(0) cube complex of the form $E G_v$, that is, it coincides with one of the hyperplanes of $E G_v$, and such a hyperplane does not cross any attaching path. Thus, the wall is contractible and separates $\cE G$ locally into two connected components by Proposition \ref{propertieshyperplane}. Since $\cE G$ is simply connected, the wall separates $\cE G$ globally into two connected components. 
 \end{proof}
 
\begin{lem}\label{WallEGseparates2}
A wall associated with an edge of $\cE G_\mathrm{\textit{bal}}$ of second type is contractible and separates $\cE G$ into two connected components. 
\end{lem}

\begin{proof} 
Let $e$ be an edge of $\cE G_\mathrm{\textit{bal}}$ whose associated hypergraph is of second type. We first use   properties of hypergraphs in $X$. 
Using  Lemma \ref{L: hypergraph tree}, we observe that the wall associated with $e$ has a structure of tree of spaces over $\Lambda_e^{\cE G}$ with fibres being (contractible) hyperplanes. The contractibility of such a wall thus  follows.

Since $\cE G_\mathrm{\textit{bal}}$ is simply connected, it is enough to prove that the associated hypergraph separates locally $\cE G_\mathrm{\textit{bal}}$ into two connected components. Therefore, we now use geometric properties of $X$ to reduce the problem to the hyperplanes in the CAT(0) cube complexes.

The only non-trivial case to consider is the preimage of a neighbourhood of a vertex of $X_\mathrm{\textit{bal}}$ contained in $\Lambda_e^{\cE G}$, that is, a point of $\Lambda_e^{\cE G}$ whose preimage in $\cE G_\mathrm{\textit{bal}}$ is a hyperplane in the associate fibre. Let $v$ be such a vertex of $X_\mathrm{\textit{bal}}$ and $H$ the hyperplane associated with an edge $e$ on the attaching path $p_{v,R}$ in $E G_v$ corresponding to a polygon of $X$.

\begin{figure}[H]
\begin{center}
\scalebox{0.6}{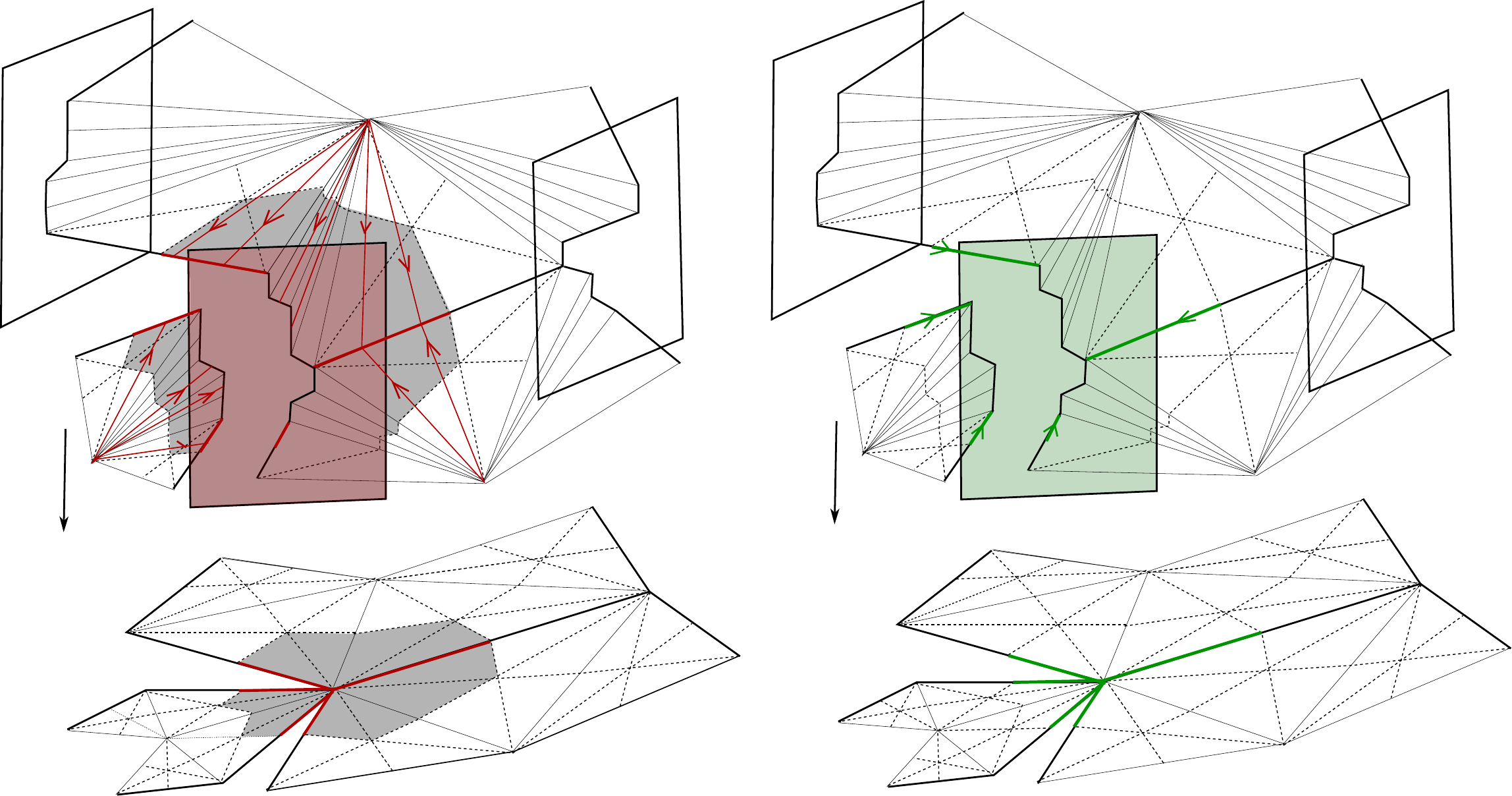}
\caption{The construction of the projection map $q_v$. The star of $v$ and its preimage in $\cE G_\mathrm{\textit{bal}}$ are represented in shaded. On the left, the various radial projections $\widetilde{R} \cap S(v) \ra \partial \widetilde{R} \cap S(v)$. On the right, the various projections $\partial \widetilde{R} \cap S(v) \ra  EG_v$. }
\label{FigureProjection}
\end{center}
\end{figure}

We now work with a finer polyhedral structure on $\cE G_\mathrm{\textit{bal}}$ obtained as follows. First consider the simplicial polygon associated with each polygon of $\cE G_\mathrm{\textit{bal}}$ (as explained in Section \ref{complexofgroupsfreeproduct}), then take its first barycentric subdivision. For this new polyhedral structure, consider the star $\mbox{st}(v)$ of $v$, that is, the union of all the simplices containing $v$. Denote by $S(v)$ the preimage of $\mbox{st}(v)$ under the projection map $p:\cE G_\mathrm{\textit{bal}} \ra X_\mathrm{\textit{bal}}$. We define a projection map $q_v: S(v) \ra EG_v$ in two steps. Let $R$ be a polygon of $X_\mathrm{\textit{bal}}$ containing $v$ and $\widetilde{R}$ its lift to $\cE G_\mathrm{\textit{bal}}$. First retract radially $\widetilde{R} \cap S(v)$ onto $\partial \widetilde{R} \cap S(v)$, then retract  $\partial \widetilde{R} \cap S(v)$ onto $\partial \widetilde{R} \cap EG_v$ (see Figure \ref{FigureProjection}).  
It is straightforward to check that these projections are compatible and define a map from $S(v)$ to $EG_v$. Furthermore, by definition of $W_e^{\cE G}$, $q_v$ restricts to a surjective map from $S(v) \setminus W_e^{\cE G}$ onto $EG_v \setminus H$. 

Finally, we use the properties of CAT(0) cube complexes to conclude. As $EG_v$ is a CAT(0) cube complex, by Proposition \ref{P: extended wall gallery}(1) the latter space is disconnected into exactly two components, so, using Proposition \ref{P: extended wall gallery}(2), is $S(v) \setminus W_e^{\cE G}$. As the preimage under $q_v$ of a path of $EG_v$ is a connected subset of $S(v)$ and $EG_v \setminus H$ has exactly two connected components, $S(v) \setminus W_e^{\cE G}$ has at most two connected components, hence it has exactly  two connected components. 
\end{proof}

We now have defined many walls on $\cE G$: lifts of walls of $X$, and extension of hyperplanes of the fibre CAT(0) cube complexes to the whole space $\cE G$. In the next section we use all these walls to define a wallspace structure on $\cE G$ that makes $\cE G$ a wallspace. We then associate a CAT(0) cube complex to such a structure.

\subsection{The wallspace and its associated CAT(0) cube complex}

In this section we combine the walls associated with edges of $X_{bal}$, Definition \ref{D: walls 1}, and the walls associated with edges of $\cE G_{bal}$, Definition \ref{D: walls 2}. This yields a wallspace structure on $\cE G_{bal}$. Figure \ref{F: 3 types of walls} shows an example of $\cE G$ with all three types of walls, together with their corresponding hypergraphs.

\begin{figure}[H]
\begin{center}
\scalebox{0.69}{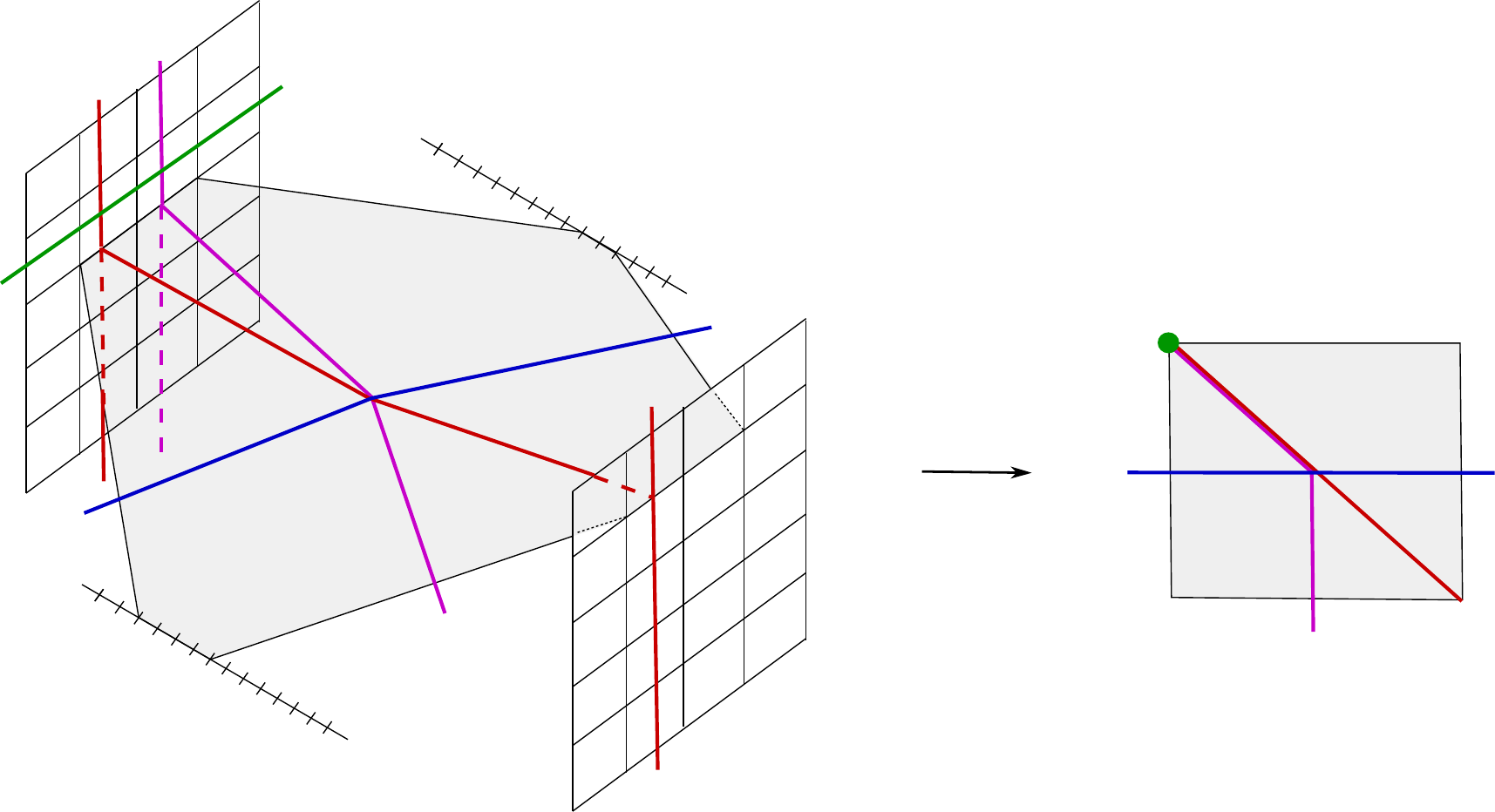}
\caption{Examples of the three types of walls of $\cE G$ (left) and their associated hypergraphs in $X$ (right), in the case $A=\mathbb{Z}^2$, $B=\mathbb{Z}$. To avoid a busy picture, we only represent the case of a polygon of $X$ with $4$ sides. Blue: Walls/Hypergraphs associated with edges of $X$. Green: Walls/Hypergraphs associated with edges of $\cE G_\mathrm{bal}$  of first type. Red and pink: Walls/Hypergraphs associated with edges of $\cE G_\mathrm{bal}$  of second type.}
 \label{F: 3 types of walls}
\end{center}
\end{figure}

\begin{definition}\label{DefinitionWalls}
We denote by $\mathcal{W}$ the family of walls of $\cE G_\mathrm{\textit{bal}}$ consisting of:
\begin{itemize}
\item the walls associated with an edge of $X_\mathrm{\textit{bal}}$,
\item the walls of first type associated with an edge of $\cE G_\mathrm{\textit{bal}}$,
\item the walls of second type associated with an edge of $\cE G_\mathrm{\textit{bal}}$.
\end{itemize}
We call an element of $\mathcal{W}$ a \textit{wall} of $\cE G_\mathrm{\textit{bal}}$. 
\end{definition}

The next result follows from combining the properties of the three types of walls that we have discussed above. 
\begin{prop}\label{EGwallstructure}
The complex $\cE G$  with the previous family of walls $\cW$ is a wallspace. \qed
\end{prop}

The wallspace $(\cE G,\cW)$ comes with an action of $G$, by setting $g\cdot W^X_e:=W^X_{g\cdot e}$ and $g\cdot W_e^{\cE G} :=W_{g\cdot e}^{\cE G}$ respectively. 

\begin{proof}[Proof of Proposition \ref{EGwallstructure}] By definition and Lemma \ref{WallXseparates} every edge in  $X_{bal}$ defines a unique wall of $\cE G$. By definition and  Lemmas \ref{WallEGseparates1} and Lemmas \ref{WallEGseparates2} an edge of $\cE G_\mathrm{\textit{bal}}$ defines a unique wall of $\cE G$.  Two vertices of $\cE G_\mathrm{\textit{bal}}$ can be joined by a finite path in the $1$-skeleton of $\cE G_\mathrm{\textit{bal}}$. As a wall separating two vertices must cross every path connecting them, the result follows. 
\end{proof}

The proof immediately implies the following useful remark.

 \begin{cor}\label{finitenumberhypercarriers2cell}
There is an upper bound on the number of walls of $\cE G$, the hypercarriers of which  contain a given polygon of $X$.\qed
\end{cor}

\begin{rmk} We note that the wall-pseudometric on $\cE G$ is a metric. Indeed, every pair of vertices of $\cE G$ is separated by a wall. To see this first consider two  vertices in the same fibre. By assumption the fibre is a CAT(0) cube complex. Then, the proof of Lemma \ref{WallEGseparates2} implies in particular that two such points are separated by at least one wall associated with a vertical edge. For vertices in two different fibres, as $X$ is a $C'(1/6)$--complex it follows from  the fact that the family of hypergraphs $\cW^X$ separates any two vertices of $X$; this last statement follows directly from  \cite[Lemma 4.3]{WiseSmallCancellation}. 
\end{rmk}

Let us now associate a CAT(0) cube complex to the wallspace $(\cE G, \cW)$, and to the wallspace $X$. A vertex of this complex is a map  $\sigma:\cW \to \cH$ sending each wall to one of the two half-spaces it defines, with some additional conditions, see \cite{ChatterjiNibloWallSpaces}. Two vertices $\sigma_1$ and $\sigma_2$ are connected by an edge if $\sigma_1$ and $\sigma_2$ differ on exactly one~wall.

\begin{definition}
 Let $C_\cW$ denote the CAT(0) cube complex associated with $(\cE G,\cW)$. 
\end{definition}

\begin{rmk}
 The action of $G$ on $\cW$ induces an action of $G$ on $ C_\cW$.
\end{rmk}

\begin{rmk}
Using the definition of the CAT(0) cube complexes associated with a wallspace \cite{ChatterjiNibloWallSpaces}, one can show that the embedding of wallspaces associated with the embedding $EG_v \hra \cE G$ yields an embedding $EG_v \hra C_\cW$ of CAT(0) cube complexes which is equivariant with respect to the map $G_v \hra G$.
\end{rmk}

Note however that there is a priori no link between the CAT(0) cube complex associated with the wallspace $(X, \cW^X)$ and $C_\cW$. 
Therefore, the results of Wise \cite{WiseSmallCancellation} that are valid for $C_X$ cannot directly be used to conclude anything about $C_\cW$.  It would technically be possible to reason solely with walls associated with the edges of $\cE G_\mathrm{\textit{bal}}$ and the associated cube complex. However, adding walls associated with edges of $X_\mathrm{\textit{bal}}$ only increases the dimension of the cube complex acted upon. We have decided to follow this approach as it seemed to us more natural from the viewpoint of the combination argument.

In the next section, we will combine results of Wise on the geometric positions of  walls of $\cW^X$ \cite[Lemma 6.4, Theorem 6.9, Theorem 11.1]{WiseSmallCancellation} with new results on the combination of such walls with the walls associated with edges of $\cE G_\mathrm{\textit{bal}}$. This will be used to prove that the wallspace structure $\cW$ on $\cE G$ is such that the induced action on the associated CAT(0) cube complex $C_\cW$ is  proper and cocompact.

\section{Cubulation theorem}

The aim of this section is to prove our main result. 

\begin{thm}\label{T: MAIN}
The action of $G$ on the CAT(0) cube complex $C_\cW$ is proper and cocompact.\label{geometricaction}
\end{thm}

The following two criteria provide information about the group action on a cube complex from  the properties of the action on the  wallspace used to define this cube complex. 

\begin{prop}[Theorem 3 in \cite{ChatterjiNibloWallSpaces}] \label{criterionproper}
Let $H$ be a group acting by isometries on a space with walls $(Y, \cW(Y))$, where $Y$ is a metric space. The $H$-action on the associated CAT(0) cube complex is proper if  for some $y \in Y$, we have $ d_{\mathcal{W}(Y)}(y,h\cdot y) \ra \infty \mbox{ when } h \ra \infty.$\qed
\end{prop}

\begin{prop}\label{criterioncocompact}
Let $H$ be a group acting on a space with walls $(Y, \cW(Y))$. The $H$-action on the associated cube complex is cocompact if and only if there exist only finitely many configurations of pairwise crossing walls of $Y$, up to the action of $H$. \qed
\end{prop}

 Therefore, we continue to study the   combination of the various type of walls underlying $C_\cW$.

\subsection{Properness}

\begin{thm}\label{properaction}
The action of $G$ on $C_\mathcal{W}$ is proper.
\end{thm}

Let us mention, once again, that we do not follow a more general approach of Wise \cite[Th. 5.50]{wise_structure_2011}. This has an advantage of a more elementary proof. Again, we combine in an appropriate way properties of the fibre  CAT(0) cube complexes and properties of the $C'(1/6)$--polygonal complex $X$. 

\begin{proof}
We first prove that the wall distance $d_{\cW}$ is proper, that is, for every vertex $x$ of $\cE G_\mathrm{\textit{bal}}$ and every integer $M \geq 0$, the set of vertices separated by at most $M$ walls from $x$ is compact. 
We give an inductive procedure to describe the set of vertices separated of $x$ by at most $M$ walls.  

Let $v$ be a vertex of $X$, $x_v$ be a vertex of $EG_v$ and $M\geq 0$ an  integer.  Let 
\[
 K_0:=\{x \in \cE G \mid x \mbox{ is a vertex of } EG_v \mbox{ and } d_{\cW(EG_v)}(x_v,x) \leq M\},                                                                                                                                                                                                                                                                                                               \]
be the   ball in $EG_v$ of radius $M$ around $x_v$. As $EG_v$ is a locally finite CAT(0) cube complex we see that $K_0$ is finite. 

We now \emph{orient}  edges $e$ of $X$ by setting one vertex of $e$ the initial and the other vertex the  terminal vertex, denoted by $i(e)$ and $t(e)$ respectively. Given an oriented edge $e$ we denote by $x_{i(e)}$ and $x_{t(e)}$ the respective attaching points of the lift $\widetilde{e}$ in $\cE G$. Let us orient each edge $e$ of $X$ at $v$ such that $i(e)=v$. Let $E_0$ be the set of those such edges with $x_{i(e)}\in K_0$.
 
Suppose we have inductively defined sets $K_0\subseteq \ldots \subseteq  K_k$ of vertices of $\cE G$ and finite sets $E_0\subseteq \ldots \subseteq E_k$ of oriented edges of $X$ such that for every such edge $e\in E_i$ we have that  $x_{i(e)}\in K_i, 0\leq i \leq k$ and $x_{t(e)}\in K_{i-1} , 0< i \leq k$. For every edge  ${e} \in E_k-E_{k-1}$  denote by $K_e$ the   ball in $EG_{t(e)}$ of radius $M + d_\cW(x_v,x_{t(e)})$ around $x_v$. Denote by $E_e$ the set of edges $e'$ of $X$ at $t(e)$ such that $i(e')\in K_e$. Set 
  \[
   K_{k+1}:=K_k\cup \bigcup_{e\in E_k} K_e,
  \]
and let 
\[
 E_{k+1}:= E_k\cup \bigcup_{e\in E_k} E_e.
\]

Again, as the various spaces $EG_v$ are CAT(0) cube complexes and $\cE G$ is locally finite, the sets $E_k$ and $K_{k}$ are finite. 

 Since $X$ is a $C'(1/6)$ polygonal complex, there exists a constant $k_M$ such that a vertex of $X$ at distance at least $k_M$ from $v$ is separated from $v$ by at least $M$ walls of $X$. Therefore and by construction, the set of vertices of $\cE G$ which are separated from $x_v$ by at most $M$ walls of $\cE G$ is contained in the set $K_{k_M}$. This set was shown to be finite, hence the claim.

Finally, let $(g_n)$ be an injective sequence of elements of $G$. Since $G$ acts properly discontinuously on $\cE G$, there are for any integer $m \geq 0$ only finitely many $n\geq 0$ such that $g_n x_v \in K_m$. Thus, $d_\cW(x_v, g_n x_v) \ra \infty$, and the result now follows from Proposition \ref{criterionproper}.
\end{proof}

Note that the proof of Theorem \ref{properaction} uses only the fact that the various fibres are locally finite CAT(0) cube complexes, and that $\cE G$ is a locally finite polyhedral complex, which follows from the fact that the fibres are locally finite and that $G$ is obtained by considering only \textit{finitely} many relators in $A*B$. In particular, redoing the whole construction in this more general framework, we obtain a proof of Theorem \ref{T: mainHaagerup}.

\begin{cor}
If $A$ and $B$ are only assumed to act \textit{properly} on locally finite CAT(0) cube complexes $EA$ and $EB$ respectively, then $G$ acts properly on $C_\cW$. \qed
\end{cor}

\subsection{Cocompactness}

Here we prove the cocompactness of the action on $C_\cW$. 

\begin{thm}\label{cocompactaction}
The action of $G$ on $C_\mathcal{W}$ is cocompact.
\end{thm}
This follows once we have shown that $\cW$ satisfies the assumptions of Proposition \ref{criterioncocompact}. 
In order to do that,  we combine, again, the properties of the fibre CAT(0) cube complexes $E_A$ and $E_B$ with the   properties of hypergraphs in the $C'(1/6)$--small cancellation polygonal complex $X$. 

In particular, we use the following properties of CAT(0) cube complexes, cf. \cite{SageevCubeComplex,GerasimovSemiSplittings}.

\begin{thm}\label{backgroundcubecomplexes}
Let $Y$ be a CAT(0) cube complex. 
\begin{itemize}
\item Given a  convex subcomplex of $Y$, its   neighbourhood, that is, the union of all the cubes meeting it, is again  convex.
\item   neighbourhoods of hyperplanes of $Y$ are  convex.
\item (Helly's theorem) Let $(Y_i)$ be a family of pairwise  convex subcomplexes of $Y$ such that any two such subcomplexes have a non-empty intersection. Then $\cap_i Y_i$ is non-empty. \qed
\end{itemize}
\end{thm}

We use the following result on the hypercarriers in $X$ of pairwise crossing walls of $\cE G_\mathrm{\textit{bal}}$. 
 
\begin{prop}\label{intersectionhypergraphs}
 Let $W_1,W_2,\ldots $ be a set of pairwise crossing walls of $\cE G$ , and let $Y_1, \ldots, Y_k $, $k \geq 3$, be the set of corresponding hypercarriers of $X$. Then the intersection $\bigcap Y_i$ is non-trivial.
\end{prop}

This result extends the following result of Wise.
\begin{lem}[cf. Theorem 6.9 of \cite{WiseSmallCancellation}]\label{L: Wise2}
 Let $\{\Lambda_1,\Lambda_2,\Lambda_3,\ldots\}$ be a set of pairwise crossing hypercarriers of $X$ defined by equivalence of diametrically opposed edges, see Section \ref{ExampleEdge}. If $\Lambda_1,  \Lambda_2, \Lambda_3, \ldots$ pairwise cross, then their common intersection contains a vertex.
\end{lem}
Let us emphasise once again, that Lemma \ref{L: Wise2} cannot directly be applied because our hypercarriers  have cutpoints,  and our far apart condition  allows  hypercarriers that differ significantly from those defined by equivalence classes of opposite edges.

\begin{proof}[Proof of Proposition \ref{intersectionhypergraphs}] We consider three cases.    If all walls $W_1,W_2,\ldots$ are associated with vertical edges in $\cE G_v$, then  $v$ is contained in the intersection of their hypercarriers.  This is  the only configuration where a wall of first type can occur. 
If all walls $W_1,W_2,\ldots$ are walls coming from $X$, associated with edges of $X_{bal}$, then Wise's Lemma \ref{L: Wise2} immediately implies the claim. 
All other configurations contain no wall of first type, and at least one wall of second type.  In this case, the proof of Wise's Lemma \ref{L: Wise2} can be extended in a straightforward way,   using our generalised notions of hypergraphs and hypercarriers. Our far apart condition is again essential. We give a full account of the arguments in Appendix \ref{Appendix hypercarriers},  see  Lemma \ref{A: intersectionhypergraphs}
\end{proof}

\begin{thm}\label{configurations} 
There is only finitely many configurations of pairwise crossing walls of $\cE G$, up to the action of $G$.
\end{thm}

\begin{proof} Let  $l$ be  the maximal length of an attaching path in the fibres of $\cE G$. 
Let $(W_i)$ be a system of pairwise crossing walls of $\cE G$ and denote by $(Y_i)$ the associated system of hypercarriers in $X$. By Lemma \ref{intersectionhypergraphs}, let $v$ be a vertex in the intersection of these hypercarriers. For each $i$, let $K_i$ be the union of all the attaching paths $p_{v,R} \subset EG_v$, where $R$ is a polygon of $Y_i$ containing $v$. We now describe the sets $K_i$, depending on the relative position of the hypergraph and the vertex $v$, as illustrated in Figure \ref{F: cubul}. 

\begin{figure}[H]
\begin{center}
\scalebox{0.8}{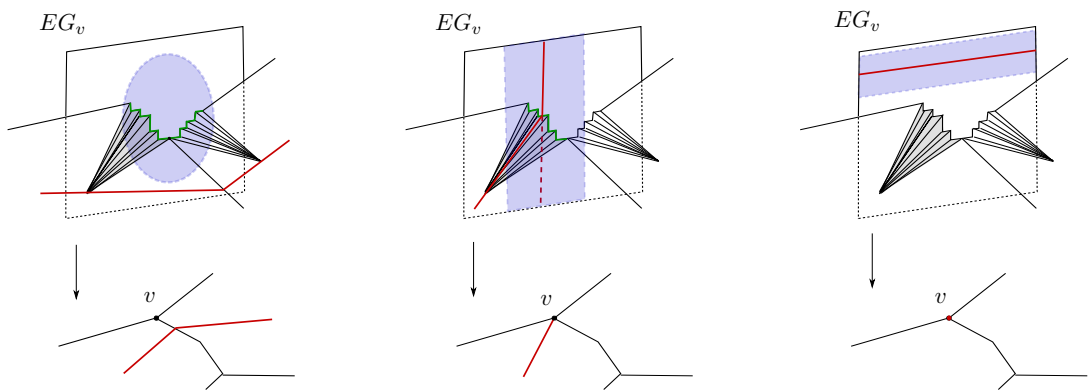}
\caption{The three possible configurations, depending on the relative position of the hypergraph associated with $W_i$ and the vertex $v$. In red: walls and hypergraphs. In green: $K_i$. In blue: $C_i$.}
 \label{F: cubul}
\end{center}
\end{figure}

If $W_i \cap EG_v$ is empty, then the vertex $v$ either belongs to an exterior arc of a polygon of $Y_i$, or $v$ belongs to a door-tree of $Y_i$. In the former case, $K_i$ consists of the single attaching path $p_{v,R}$. We then denote by $u$ the starting vertex of $p_{v,R}$ in $EG_v$. In the latter case, all the polygons $R_j$ of $Y_i$ containing $v$ share a common edge containing $v$. Then  $K_i$ consists of the union of all the attaching path $p_{v,R_j}$. These paths intersect in one vertex in $ EG_v$, that we denote by $u$. In both cases, let $C_i$ be the   $2l$-ball around $u$.  It follows that $K_i$ is contained in $C_i$. 

If  $W_i \cap EG_v$ is nonempty, then $W_i$ is a wall associated with a vertical edge of $\cE G_\mathrm{bal}$  of first or of second type. If $W_i$ is a wall of first type associated with an edge of $\cE G_v$, then $W_i\cap EG_v=W_i$, and $K_i$ is empty. Then let $C_i$ be the   $2l$-neighbourhood of the hyperplane corresponding to $W_i$.  

If $W_i$ is a wall of second type associated with an edge of of $E G_v$, then let $C_i$ be the $2l$-neighbour\-hood of the hyperplane $W_i\cap EG_v$. By definition the attaching path of any polygon of $Y_i$ containing $v$ must cross the hyperplane $W_i \cap EG_v$. Thus the subset $K_i$ is contained in $C_i$.

For two given indices $i$ and $j$, we have that $C_ i\cap C_j \neq \varnothing$.
 Indeed, if $W_i$ and $W_j$ cross in $EG_v$ this is immediate. If $W_i$ and $W_j$ do not cross in $EG_v$, then choose a cell $R$ of $X_\mathrm{\textit{bal}}$ whose preimage in $\cE G_\mathrm{\textit{bal}}$ contains a point of $W_i \cap W_j$. Choose a vertex $w$ of $R$ other than $v$, and consider a   geodesic between $v$ and $w$. By Proposition \ref{T: properties hyper}, such a geodesic is contained in $Y_i \cap Y_j$. In particular, the unique edge of that geodesic containing $v$ is in $Y_i \cap Y_j$, which implies that 
 $C_i \cap C_j \neq \varnothing$. The various subcomplexes $C_i$ are  convex by Theorem \ref{backgroundcubecomplexes}. Thus, Helly's theorem implies that the intersection $\cap_i C_i$ is non-empty. Let $w$ be a vertex in this intersection, and let $C$ be the   $4l$-ball around $w$. Note that, as $EG_v$ is a locally 
finite  CAT(0) cube complex, the set $C$ is finite.

 Let us now consider two cases. First suppose all hypergraphs $W_i$ intersect in $EG_v$. Then for each hypergraph $W_i$ there is an edge $e_i$ contained in $C$ such that the hyperplane associated with $e_i$ equals $\widetilde{\Lambda_i}\cap EG_v$. Therefore the information that is necessary to reconstruct such a situation is contained in the finite subset $C$ of $EG_v$. 
 
 Now suppose that at least one hypergraph $W_j$ does not contain $v$, that is $W_j\cap EG_v=\varnothing$. First note that there is no wall of first type in this situation. Then observe that $C$ contains $K_j$. Hence $C$ contains the attaching paths $p_{v,R}$ contained in $K_i\cap K_j$ for all $i$. 
 Thus, $C$ contains an attaching path associated with $W_i$ for all $i$.  Again, as $C$ is finite, there are only finitely many attaching paths contained in $C$, and therefore the information that is necessary to reconstruct such a situation is contained in the finite subset $C$ of $EG_v$. 
 
Since the action of $A$ on $EA$ (resp. of $B$ on $EB$) is cocompact, choose a compact subcomplex $K_{A}$ (resp. $K_B$) of $EG_{v_A}$ (resp. $EG_{v_B}$) which contains an $A$-translate (resp. a $B$-translate) of every   $4l$-ball of $EG_{v_A}$ (resp. $EG_{v_B}$).  

Let $g$ be an element of $G$ which sends $C$ to a subcomplex of $K_A \cup K_B$. 
In the first case above, as the fibres are locally finite there are only finitely many possibilities for the walls $(gWi)$. 
In the second above case, let $\mathcal{P}$ be the set of polygons of $\cE G$ such that one of their attaching paths meet $K_A$ or $K_B$. This set is finite since the action of $A$ on $EA$ (resp. of $B$ on $EB$) is properly discontinuous. As $\cP$ is finite, and  by Corollary \ref{finitenumberhypercarriers2cell}, there are only finitely many possibilities for the walls $(gW_i)$. Hence, in total there are only finitely many possibilities for the walls $(gW_i)$.
\end{proof}

Theorem \ref{cocompactaction} now follows from Proposition \ref{criterioncocompact} and Theorem \ref{configurations}.\qed\medskip

\appendix

\section{Appendix: Small cancellation polygonal complexes}\label{Appendix A}

Let us denote by  $X$  a $C'(1/6)$ polygonal complex. 
Here, we study the   geometry of $X$. The results can then be applied to the $C'(1/6)$ polygonal complex defined in Section \ref{complexofgroupsfreeproduct}.

\subsection{Classification of disc diagrams}
\begin{definition}[disc diagram over $X$, reduced disc diagrams, arcs] A disc diagram $D$ \textit{over the $C'(1/6)$ polygonal complex $X$} is a contractible planar polygonal complex endowed with a   map  $D \ra X$ which is an embedding on each polygon. A disc diagram $D$ over $X$ is called \textit{reduced} if no two distinct polygons of $D$ that share an edge are sent to the same polygon of $X$.

For a disc diagram $D$, we denote by $\partial D$ its boundary and $\mathring{D}$ its interior. The \textit{area} of a diagram $D$, denoted Area$(D)$, is the number of polygons of $D$. 
For a polygon $R$ of $D$, the intersection $\partial R \cap \partial D$ is called the \textit{outer component} of $R$ (and the \textit{outer path} if such an intersection is connected), the closure of $\partial R \cap \mathring{D}$ is called the \textit{inner component} of $R$ (and the \textit{inner path} if such an intersection is connected).

A diagram is called \textit{non-degenerate} if its boundary is homeomorphic to a circle, \textit{degenerate} otherwise. An \textit{arc} of $D$ is a path of $D$ whose interior vertices have valence $2$ and whose boundary vertices have valence at least $3$. Such an arc is called \textit{internal} if its interior is contained in $\mathring{D}$, \textit{external} if  the arc is fully contained in $\partial D$.
\end{definition}

We have the following fundamental result:

\begin{thm}[Lyndon--van Kampen]
Every loop of $X$ is the boundary of a reduced disc diagram. \qed
\end{thm} 

All  disc diagrams considered in this Appendix will be reduced without further notice. We now present a classification theorem for reduced disc diagrams.

\begin{definition}[ladder] A reduced disc diagram $D$  of $X$ is a \textit{ladder} if it can be written as a union $D = c_1 \cup \ldots \cup c_n$, where the $c_i$ are edges or polygons of $X$ and such that:
\begin{itemize}
\item $D \setminus c_1$ and $D \setminus c_n$ are connected,
\item $D\setminus c_i$ has exactly two connected components for $1 < i < n$.
\end{itemize}
\end{definition}

\begin{definition}[shell, spur]
Let $D$ be a reduced disc diagram of $X$. A \textit{shell} of $D$ is a polygon of $D$ such that $\partial R \cap \partial D$ is connected and whose inner path is the concatenation of at most $3$ internal arcs of $D$. 
A \textit{spur} of $D$ is an edge of $D$ with a vertex of valence $1$.
\end{definition}

\begin{rmk}
Note that the internal arcs involved in the previous definition are automatically sent to  pieces of $X$ by the properties of a reduced disc diagram.
\end{rmk}

The following is the fundamental result of small cancellation theory (a version of the well-known Greendlinger 
Lemma, see Theorem 4.5 in \cite[Chapter V.4]{LyndonSchupp}). This version follows directly from Theorem 9.4 of \cite{McCammondWiseFansLadders}.

\begin{thm}[Classification Theorem for disc diagrams]\label{classification}
Let $D$ be a reduced disc diagram of $X$. Then either:
\begin{itemize}
\item $D$ consists of a single vertex, edge or polygon,
\item $D$ is a ladder,
\item $D$ contains at least three shells or spurs.\qed
\end{itemize}
\end{thm}

The proof of this theorem is based on a negative curvature phenomenon  described via  a   version of Gau\ss-Bonnet's Theorem. We now explain this theorem as it   is used later. 

\begin{definition}[corner, disc diagram with angles]
A \textit{corner} of a (reduced) disc diagram $D$ of $X$ is a pair $(v,R)$ where $v$ is a vertex of $D$ and $R$ a polygon containing it. We denote by Corner$(v)$ (resp. Corner$(R)$) the set of corners of the form $(v,R')$ (resp. $(v', R)$). 

We say that $D$ is a disc diagram \textit{with angles} if each corner $c$ is assigned an \textit{angle} $\angle(c) \geq 0$.

For a vertex $v$ of $D$, we define its \textit{curvature}:
$$\kappa(v) = 2\pi - \pi \cdot \chi(\mbox{link}(v)) - \sum_{c \in \mathrm{Corner}(v)} \angle(c).$$

For a polygon $R$ of $X$, we define its \textit{curvature}:
$$\kappa(R) =  \sum_{c \in \mathrm{Corner}(R)} \angle(c) - \pi \cdot |\partial R| + 2\pi.$$
\end{definition} 

\begin{thm}[  Gau\ss-Bonnet Theorem] For a (reduced) disc diagram of $X$ with angles, we have:
$$\underset{v \mathrm{ ~vertex ~of~ }D}{\sum} \kappa(v) + \underset{R \mathrm{~polygon~of~} D}{\sum} \kappa(R) = 2\pi.$$\qed
\label{GaussBonnet}
\end{thm}
\subsection{Hypercarriers embed}

Galleries were defined in Definition \ref{D: gallery}. 
We prove the following result, which generalises a result of Wise \cite[Lemma 3.11]{WiseSmallCancellation}:

\begin{prop}\label{simplyconnectedembedded}
Let $\cC$ be a gallery. Then its hypercarrier $Y_\cC$ is connected and simply connected and the map $i_\cC:Y(\cC) \ra X$ is en embedding. 
\end{prop}

The proof of this proposition is using all three properties of a gallery, in particular the far apart condition. Extending the arguments of \cite{WiseSmallCancellation} in a straight-forward way, we give the detailed proof below.

\begin{lem}\label{arcexterieur}
Let $\cC$ be a gallery and let $R_{\{\tau_1, \tau_2\}} $ be a polygon of $\cC$. Let $P_1,P_2 \subset \partial R_{\{\tau_1, \tau_2\}}$  be distinct paths such that the concatenations $\tau_1  P_1$ and $ \tau_2  P_2$ are pieces of $X$. Then \emph{no} connected component of $\partial R \setminus (\tau_1P_1 \cup \tau_2 P_2)$  is  covered by a single piece, and  $P_1$ and $P_2$ are disjoint.
\end{lem}

\begin{proof}
If a connected component $C$ of $\partial R \setminus (\tau_1P_1 \cup \tau_2 P_2)$  is  covered by a single piece, the path from $\tau_1$ to $\tau_2$ covering $C$   consists of  at most three pieces. This contradicts the far apart condition.  
If $P_1$ and $P_2$ intersect, the path covering $\tau_1P_1$ and $\tau_2P_2$ consists of at most two pieces, again contradicting the far apart condition.
\end{proof}

\begin{definition}[canonical decomposition of a $2$-cell, exterior arcs, door-trees]\label{canonicaldecomposition}
Let $\cC$ be a gallery with hypercarrier $Y_\cC$ and let $R_{\{\tau_1, \tau_2\}} $ be a polygon of $\cC$. Let $P_1, P_1', P_2, P_2' \subset \partial R_{\{\tau_1,\tau_2\}}$ be maximal paths such that the concatenations $\tau_1P_1, \tau_1P_1', \tau_2P_2, \tau_2,P_2'$ are pieces.

Let $A, A' \subset \partial R_{\{\tau_1,\tau_2\}}$ be the paths joining the extremities of $P_1, P_2$ and $P_1',P_2'$, called the \textit{exterior arcs} of $R_{\{\tau_1, \tau_2\}} $.

The union of all the paths of the form $\tau_1P_1$ and $\tau_1P_1'$, where $R$ runs over the polygons of $\cC$ containing $\tau_1$ as door, is a tree, called the \textit{door-tree} associated with the door $\tau_1$.
\end{definition}

By definition of $Y_\cC$, no edge of $A$ or $A'$ is identified to the edge of a distinct polygon of  $Y_\cC$ which is glued to $R_{\{\tau_1,\tau_2\}}$ along either $\tau_1$ or $\tau_2$. This implies in particular that two distinct polygons of $Y_\cC$ sharing a door of $\cC$ are sent to different polygons of $X$. As the map $i_\cC: Y_\cC \ra X$ is already an immersion at the level of the $1$-skeleton, the following follows: 

\begin{cor}
Let $\cC$ be a gallery of $X$. Then the map $i_\cC: Y_\cC \ra X$ is an immersion. \qed
\end{cor}

\begin{lem}\label{auplusdeuxpieces}
Let $\cC$ be a gallery of $X$. Let $R$ be a polygon of $X$ meeting $i_\cC(Y_\cC)$ which does not contain a door of $\cC$. Let $P$ be a path of $\partial R \cap i_\cC(Y_\cC)$ which admits a lift to $Y_\cC$ under $i_\cC$. Then $P$ is covered by the concatenation of at most two pieces.
\end{lem}

\begin{proof}
 Lemma \ref{arcexterieur} implies that $P$  cannot cover a complete exterior arc $A$. Thus, either $P$ is a proper subpath of $A$, or $P$ intersects exactly two polygons of $\cC$. In the former case, $P$ is covered by one piece, in the latter case $P$ is covered by two pieces.
\end{proof}

Note that we have, so far, not used the connectedness nor the coherence condition in the definition of a gallery, see Definition \ref{D: gallery}.

\begin{proof}[Proof of Proposition \ref{simplyconnectedembedded}]
The fact that $Y_\cC$ is connected is a direct consequence of the connectedness condition.

We say that a path $P$ of $Y_\cC$ is \textit{essential} if it is a loop representing a non-trivial element of the fundamental group of $Y_\cC$, or if it is a path with distinct extremities which are sent to the same vertex of $X$. In the latter case, we call such a vertex of $X$ the \textit{unique singular vertex of $i_\cC(P)$}. The proposition amounts to proving that there exists no essential path in  $Y_\cC$.

We reason by contradiction. Let $P$ be such an essential path of $Y_\cC$. Since $X$ is simply-connected, the loop $i_\cC(P)$ is the boundary of a disc diagram $D$. Notice first that $D$ cannot be a single vertex or edge. Without loss of generality, we can assume that the number of polygons of $D$ is minimal among such diagrams. In particular, $D$ is non-degenerate and each path of its boundary $i_\cC(P)$ that does not contain the singular vertex of $i_\cC(P)$ lifts to a path of $P \subset Y_\cC$. 

First suppose that $D$ is a single polygon. By hypothesis on $P$, $D$ cannot be contained in $i_\cC(Y_\cC)$. Let us decompose the boundary of $D$ as the union of two paths $P_1$ and $P_2$ neither of which contains the singular vertex of $i_\cC(Y_\cC)$ in their interior. Both paths $P_1$ and $P_2$ thus lift to paths of $Y_\cC$. By Lemma \ref{arcexterieur}, this implies that $P_1$ and $P_2$ can be covered by the concatenation of two pieces, and so the boundary of $D$ is covered by fours pieces, contradicting the condition $C'(1/6)$.

By the classification theorem \ref{classification}, this implies that  the disc diagram $D$ contains at least two shells, and we can choose one of these shells, say $R$, so that its outer path does not contain the singular vertex of $i_\cC(Y_\cC)$ in its interior. Such a shell must be contained in $i_\cC(Y_\cC)$,  for otherwise Lemma \ref{auplusdeuxpieces} would imply that  $R \cap \partial D$ is covered by two pieces, making the boundary of $R$ covered by five pieces, a contradiction with condition $C'(1/6)$. Thus $R \subset i_\cC(Y_\cC)$ and we can push the path $P$ through the lift of $R$ in $Y_\cC$ to obtain a new essential path, the image of which in $X$ is the image in $X$ of the boundary of the disc diagram $D \setminus R$. As such a diagram contains strictly fewer polygons than $D$, we get a contradiction. 
\end{proof}

\begin{cor}\label{A:hypergraphtree}
For every gallery $\cC$, the associated hypergraph $\Lambda_\cC$ is a tree which embeds in $X$.
\end{cor}

\begin{proof}
It is enough by Proposition \ref{simplyconnectedembedded} to see that the associated hypercarrier $Y_\cC$ retracts by deformation onto $\Lambda_\cC$. Such a deformation is easily defined using the canonical decomposition of a polygon of $\cC$ introduced in Definition \ref{canonicaldecomposition}.
\end{proof}

\begin{rmk}[minimal ladder between two  simplices of a hypercarrier]
Let $\cC$ be a gallery and $\tau$ and $\tau'$ be two simplices of $Y_\cC$ that are not contained in the same door-tree of $Y_\cC$. There exists a unique non-degenerate ladder of minimal area containing $\tau$ and $\tau'$, which we call the \textit{(minimal) ladder of $Y_\cC$ between $\tau$ and $\tau'$}.
\end{rmk}

\subsection{Convexity of hypercarriers}

Here we prove the following: 

\begin{prop}\label{A: lyconvex}
Let $\cC$ be a gallery. Then the subcomplex $Y_\cC$ of $X$ is  convex, that is, a   geodesic between two vertices of $Y_\cC$ is contained in $Y_\cC$.
\end{prop}

We will prove that proposition by contradiction. Let us assume that there exists a   geodesic $P$ between two vertices of $Y_\cC$ and such that $P$ is not contained in $Y_\cC$. Let  $Q$ be a path of $Y_\cC$ joining the two extremities of $P$. The union of $P$ and $Q$ yields a loop of $X$, and thus there exists a disc diagram with such a loop as boundary. We choose $P, Q$ and $D$ in such a way that $(|P|, \mbox{Area}(D) )$ is minimal for the lexicographic order. In particular, $P$ does not cross the hypergraph $\Lambda_\cC$. We now study separately three cases. 

\begin{lem}
The diagram $D$ cannot consist of a single polygon.
\label{convex1}
\end{lem}

\begin{proof}By contradiction, suppose that $D$ consists of a single polygon $R$ of $X$. Since $R$ is not contained in $Y_\cC$ by assumption, the path  $Q= R \cap Y_\cC$ is covered by at most two pieces by Lemma \ref{auplusdeuxpieces}. Thus, condition $C'(1/6)$ implies that $|Q| < \frac{1}{2} |\partial D|$, hence $|P| > \frac{1}{2} |\partial D| > |Q|$, contradicting the fact that $P$ is a   geodesic. 
\end{proof}

\begin{lem} The diagram $D$ cannot contain three shells.
\label{convex2}
\end{lem}

\begin{proof} By contradiction, suppose that $D$ contains three shells. We can thus choose one of them, say $R$, whose outer boundary is contained either in $P$ or in $Q$.

First assume that such an outer path is contained in $P$. We can thus push $P$ through $R$  to get a new path $P'$ such that the union $P' \cup Q$ is the boundary of the disc diagram $D \setminus R$. Let $L$ be the concatenation of the inner arcs of $R$. Since $R$ is a shell, the $C'(1/6)$--condition implies $|L| < \frac{1}{2} |\partial R |$, hence $|P'| < |P|$, a contradiction. 

Assume now that this outer path of $R$ is contained in $Q$. First notice that $R$ has to be contained in $Y_\cC$, for otherwise such an arc would be covered by two pieces by Lemma \ref{auplusdeuxpieces} and since $R$ is a shell the whole of $\partial R$ would be covered by five pieces, contradicting the $C'(1/6)$--condition. Thus $R \subset Y_\cC$ and we can push $Q$ through $R$ to obtain a new path $Q'$ of $Y_\cC$ such that $P \cup Q'$ is the boundary of the disc diagram $D \setminus R$, contradicting the minimality of $D$.
\end{proof}

\begin{lem} The disc diagram $D$ cannot be a ladder. 
\label{convex3}
\end{lem}

\begin{proof}
By contradiction, suppose that $D$ is a (non-trivial) ladder. The minimality assumption implies that $D$ is non-degenerate. Let us write  $D=R_1 \cup R_2 \cup \ldots $ and let $P_1$ be the portion of $P$ contained in $R_1$, and $P_2$ its complement in $R_1$. 

We can push $P$ through $R_1$ to obtain a new path $P_1'$. Since $P$ does not cross $\Lambda_\cC$, $R_1$ is not contained in $Y_\cC$ and thus $R_1 \cap Y_\cC$ is covered by two pieces by Lemma \ref{auplusdeuxpieces}. As $R_1 \cap R_2$ is also a piece, it follows that $P_2$ is covered by three pieces, and condition $C'(1/6)$ now implies $|P_2| < \frac{1}{2} |\partial R_1 | < |P_1|$, a contradiction. 
\end{proof}

\begin{proof}[Proof of Proposition \ref{A: lyconvex}]
This follows from Lemmas \ref{convex1}, \ref{convex2}, \ref{convex3}, together with the classification theorem for disc diagrams \ref{classification}.
\end{proof}

\begin{cor}\label{polygonsareconvex} Polygons of $X$ are  convex. \qed
\end{cor}

\subsection{Intersections of hypercarriers} \label{Appendix hypercarriers}
 
In this section, we extend the following results of Wise.

 \begin{lem}[cf. Lemma 6.4 of \cite{WiseSmallCancellation}]\label{L: Wise1}
Let $Y_1$, $Y_2$ and $Y_3$ be hypercarriers of $X$ defined by equivalence of diametrically opposed edges, see Section \ref{ExampleEdge}. If $Y_1$, $Y_2$ and $Y_3$ pairwise cross, then their common intersection is non-trivial.
\end{lem}

\begin{lem}[cf. Theorem 6.9 of \cite{WiseSmallCancellation}]\label{AL: Wise2}
 Let $\{Y_1,Y_2,Y_3,\ldots\}$ be a set of pairwise crossing hypercarriers of $X$ defined by equivalence of diametrically opposed edges, see Section \ref{ExampleEdge}. If $Y_1$, $Y_2$ and $Y_3$ pairwise cross, then their common intersection contains a vertex.
\end{lem}

Again, the proofs are extensions of Wise's original proofs, the small difference being related to cut-points in hypercarriers. The generalised hypercarriers coming from the far apart condition play no particular role here, as we treat them with  the results of the  previous sections. However,  the corresponding results of \cite{WiseSmallCancellation} are not sufficient.

\begin{lem}\label{intersection3hypergraphs}
Let $Y_1, Y_2, Y_3$ be three pairwise crossing hypercarriers of $X_\mathrm{\textit{bal}}$. Then the intersection $Y_1 \cap Y_2 \cap Y_3$ contains a vertex. 
\end{lem}

\begin{proof}
 We can restrict to the case where $Y_1 \cap Y_2 \cap Y_3$ does not contain a polygon. First choose cells $\sigma_{1,2} \subset Y_1 \cap Y_2$, $\sigma_{2,3} \subset Y_2 \cap Y_3$ and $\sigma_{3,1} \subset Y_3 \cap Y_1$ of maximal dimension  such that the preimage of $\sigma_{i,j}$ in $\cE G_\mathrm{\textit{bal}}$ contains a point of $W_i \cap W_j$. The cell $\sigma_{i,j}$ is either a polygon $R_{i,j}$ or a vertex $v_{i,j}$. In the former case, the hypergraphs of $Y_i$ and $Y_j$ intersect in the apex of $R_{i,j}$, in the latter, the fibre over $v_{i,j}$ contains both, a hyperplane of $Y_i$, and a hyperplane of $Y_j$. 

If two of these cells $\sigma_{i,j}$ coincide, then it defines a cell in $Y_1 \cap Y_2 \cap Y_3$. Suppose this is not the case.  For pairwise distinct $i,j,k \in \{1,2,3\}$, consider the minimal ladder $L_i$ in $Y_i$ between $\sigma_{i,j}$ and $\sigma_{i,k}$.  We choose such a configuration in such a way that the number of polygons in $L_1 \cup L_2 \cup L_3$ is minimal. Denote by $\lambda_1\subset L_1$ the portion of the \emph{hypergraph $\Lambda_1$ associated with $Y_1$} which is the geodesic of $\Lambda_1$ 
joining the barycentres of $\sigma_{1,2}$ and $\sigma_{3,1}$, and define similarly $\lambda_2\subset L_2$ and $\lambda_3\subset L_3$. Subdivide the polygons of $L_1 \cup L_2 \cup L_2$  in a minimal way such that $\lambda_1 \cup \lambda_2 \cup \lambda_3$ defines a triangle of the $1$-skeleton of $X$. Denote by $v_{i,j}$ the vertex associated with the cell $\sigma_{i,j}$. Consider now a reduced disc diagram whose boundary path is $\lambda_1 \cup \lambda_2 \cup \lambda_3$. We now endow $D$ with a structure of disc diagram 
with 
angles:
\begin{itemize}
\item If $\sigma_{i,j}$ is a polygon $R_{i,j}$, the corner at the vertex corresponding to $R_{i,j}$ is given the angle $\frac{(n_{i,j}-3)\pi}{3}$, where $n_{i,j}$ is the number of sides of the polygon of $D$ containing that vertex. Note that by minimality of the number of polygons in $L_1 \cup L_2 \cup L_3$, we necessarily have $n_{i,j} \geq 4$. If $\sigma_{i,j}$ is a vertex $v_{i,j}$, then by minimality of the number of polygons of $L_1 \cup L_2 \cup L_3$, there are at least two distinct polygons of $D$ containing $v_{i,j}$.
\item Each other corner of $D$ relying on an edge of $\partial D$ is given an angle $\frac{\pi}{2}$. 
\item All remaining corners are given an angle $\frac{2\pi}{3}$.
\end{itemize}
It is straightforward to check that with such a choice of angles, every polygon and every vertex of $D$ has non-positive curvature  by the $C'(1/6)$--condition, apart maybe from the  the vertices corresponding to the various $R_{i,j}$. The curvature at each such vertex being at most $\frac{2\pi}{3}$, it must be exactly $\frac{2\pi}{3}$ by the   Gauss Bonnet Theorem \ref{GaussBonnet} (in particular, each $\sigma_{i,j}$ is a polygon $R_{i,j}$). Thus, there is no vertex or polygon with negative curvature. In particular, since an internal polygon of $D$ would have at least $7$ sides by the $C'(1/6)$--condition, and since such a cell would have negative curvature, $D$ contains no internal polygon. Thus the image of $D$ is contained in $L_1 \cup L_2 \cup L_3$ and $L_1 \cap L_2 \cap L_3$, hence $Y_1 \cap Y_2 \cap Y_3$, must be non-empty. 
\end{proof}

\begin{lem}\label{A: intersectionhypergraphs}
Let $Y_1, \ldots, Y_k $, $k \geq 3$, be a set of pairwise crossing hypercarriers of $X_\mathrm{\textit{bal}}$. Then the intersection $\bigcap Y_i$  contains a vertex.
\end{lem}

\begin{proof} We again use the methods we have developed in Section \ref{S: extending the hyperplanes of fibres} and this Appendix to extend the original arguments of Wise's proof of Lemma \ref{AL: Wise2}. 
We prove the result by induction on $k \geq 3$, the case $k=3$ being Lemma \ref{intersection3hypergraphs}. For a subset $S$ of $I:=\{1, \ldots, k\}$, we denote by $Y_S$ the intersection of the hypergraphs $Y_i$ for $i \in S$. 

By the induction hypothesis, the intersections $Y_{I - \{1\}} $, $Y_{I - \{2\}} $ and $Y_{I - \{3\}} $ contain a vertex, denoted respectively $v_1, v_2$ and $v_3$. Choose a   geodesic between $v_i$ and $v_j$ for $1 \leq i\neq j \leq 3$, which we denote $P_{i,j}$. By Proposition \ref{A: lyconvex}, we have that $P_{i,j} \subset Y_{I - \{i,j \}} \subset Y_k$. 

 If $Y_i$ is a hypercarrier defined by equivalence of diametrically opposed edges, see Section \ref{ExampleEdge}, its boundary $\partial Y_i$ is the disjoint union of two trees, $\partial_+Y_i$ and $\partial_-Y_i$ and $Y_i$ retracts by deformation on each of these trees. The situation is slightly different  here since vertices can be local cut-points of $Y_i$. However, by reasoning separately on the closure of each component of $Y_i$ with its cut-points removed, we can write $\partial Y_i$ as the union of two trees $\partial_+ Y_i$ and $\partial_- Y_i$ whose intersection is contained in the set of cut-points of $Y_i$ and such that $Y_i$ retracts by deformation on each of these two trees.

We now consider two cases, depending on the relative position of $v_1$, $v_2$ and $v_3$ inside $Y_k$. First assume that $v_1$, $v_2$ and $v_3$ are contained in the same boundary component of $Y_k$, say $\partial_+ Y_k$. We can thus replace the paths $P_{i,j}$ by immersed paths $P_{i,j}'$ between $v_i$ and $v_j$, and which is contained in the tree $\partial_+ Y_k$. In particular, the intersection $P_{1,2}' \cap P_{2,3}' \cap P_{3,1}'$ contains a vertex, which is thus contained in $Y_{I - \{1,2 \}} \cap Y_{I - \{2,3 \}} \cap Y_{I - \{3,1 \}} = Y_I$.

Let us now assume that $v_1$ and $v_2$ are contained in the same component $\partial_+ Y_k$ and $v_3$ is contained in $\partial_- Y_k$. For $i=1,2$, consider the minimal ladder $L_{i,3} \subset Y_k$ between $v_i$ and $v_3$ and define the path $P_{i,3}' := L_{i,3} \cap \partial_+ Y_k$. Consider the sequence of doors between $v_3$ and $v_1$, and between $v_3$ and $v_2$. If these sequences do not share the same initial door, then $v_3$ belongs to one of the exterior arcs of some polygon $R$ of $Y_k$. Since both doors of $R$ also belong to $Y_{I - \{3\}}$, this subcomplex contains a subpath of $\partial R$ of length $\frac{|\partial R|}{2}$ by Corollary \ref{polygonsareconvex}. This implies that $R \subset Y_{I - \{3\}}$ by Lemma \ref{auplusdeuxpieces}, and thus the other exterior arc of $R$ is contained $P_{1,2}' \cap Y_{I - \{3\}} \subset Y_I$. Otherwise consider the last door in this initial common subsequence. Then one of the vertices of this door is contained in $P_{1,2}' \cap Y_{I - \{3\}} \subset Y_I$.
\end{proof}

\bibliographystyle{hep}
\bibliography{CubulationFreeProduct}

\def\cprime{$'$} \def\cprime{$'$}
\begin{thebibliography}{Mar14b}

\bibitem[Ago13]{AgolVirtualHaken}
I.~Agol, \textsl{ The virtual {H}aken conjecture},
\newblock Doc. Math. \textbf{ 18}, 1045--1087 (2013),
\newblock With an appendix by I. Agol, D. Groves, and J. Manning.

\bibitem[AS14]{ArzhantsevaSteenbockRipsConstruction}
G.~Arzhantseva and M.~Steenbock, \textsl{ Rips construction without unique
  product},
\newblock (2014),
  \href{http://xxx.lanl.gov/abs/arXiv:1407.2441}{arXiv:1407.2441}.

\bibitem[BH99]{BridsonHaefliger}
M.~R. Bridson and A.~Haefliger,
\newblock \textsl{ Metric spaces of non-positive curvature},
\newblock Springer-Verlag, Berlin, 1999.

\bibitem[CN05]{ChatterjiNibloWallSpaces}
I.~Chatterji and G.~Niblo, \textsl{ From wall spaces to {$\rm CAT(0)$} cube
  complexes},
\newblock Internat. J. Algebra Comput. \textbf{ 15}(5-6), 875--885 (2005).

\bibitem[Cor92]{CorsonComplexesofGroups}
J.~M. Corson, \textsl{ Complexes of groups},
\newblock Proc. London Math. Soc. (3) \textbf{ 65}(1), 199--224 (1992).

\bibitem[EJ10]{edjvet_nonsingular_2010}
M.~Edjvet and A.~Juh{\'a}sz, \textsl{ Nonsingular equations over groups {II}},
\newblock Comm. Algebra \textbf{ 38}(5), 1640--1657 (2010).

\bibitem[EJ11]{edjvet_nonsingular_2011}
M.~Edjvet and A.~Juh{\'a}sz, \textsl{ Non-singular equations over groups {I}},
\newblock Algebra Colloq. \textbf{ 18}(2), 221--240 (2011).

\bibitem[Ger97]{GerasimovSemiSplittings}
V.~N. Gerasimov,
\newblock Semi-splittings of groups and actions on cubings,
\newblock in \textsl{ Algebra, geometry, analysis and mathematical physics},
  pages 91--109, 190, Izdat. Ross. Akad. Nauk Sib. Otd. Inst. Mat.,
  Novosibirsk, 1997.

\bibitem[GMS15]{GruberMartinSteenbockFiniteIndex}
D.~Gruber, A.~Martin and M.~Steenbock, \textsl{ Finite index subgroups without
  unique product in graphical small cancellation groups},
\newblock Bull. London Math. Soc. \textbf{ 47}(4), 631--638 (2015).

\bibitem[GS14]{dominik_infintely_2014}
D.~Gruber and A.~Sisto, \textsl{ Infinitely presented graphical small
  cancellation groups are acylindrically hyperbolic},
\newblock (2014),
  \href{http://xxx.lanl.gov/abs/arXiv:1408.4488}{arXiv:1408.4488}.

\bibitem[Hae91]{HaefligerOrbihedra}
A.~Haefliger,
\newblock Complexes of groups and orbihedra,
\newblock in \textsl{ Group theory from a geometrical viewpoint ({T}rieste,
  1990)}, pages 504--540, World Sci. Publ., River Edge, NJ, 1991.

\bibitem[HK01]{higson_etheory_2001}
N.~Higson and G.~Kasparov, \textsl{ {$E$}-theory and {$KK$}-theory for groups
  which act properly and isometrically on {H}ilbert space},
\newblock Invent. Math. \textbf{ 144}(1), 23--74 (2001).

\bibitem[HP98]{HaglundPaulinWallSpaces}
F.~Haglund and F.~Paulin,
\newblock Simplicit\'e de groupes d'automorphismes d'espaces \`a courbure
  n\'egative,
\newblock in \textsl{ The {E}pstein birthday schrift}, volume~1 of \textsl{
  Geom. Topol. Monogr.}, pages 181--248 (electronic), Geom. Topol. Publ.,
  Coventry, 1998.

\bibitem[HW08]{HaglundWiseSpecial}
F.~Haglund and D.~T. Wise, \textsl{ Special cube complexes},
\newblock Geom. Funct. Anal. \textbf{ 17}(5), 1551--1620 (2008).

\bibitem[HW12a]{HaglundWiseCombination}
F.~Haglund and D.~T. Wise, \textsl{ A combination theorem for special cube
  complexes. {I}},
\newblock Ann. of Math. \textbf{ 176}(3), 1427--1482 (2012).

\bibitem[HW12b]{HsuWiseAmalgams}
T.~Hsu and D.~T. Wise, \textsl{ Cubulating malnormal amalgams},
\newblock (2012), \href{http://xxx.lanl.gov/abs/preprint}{preprint}.

\bibitem[HW14]{hruska_finiteness_2014}
G.~C. Hruska and D.~T. Wise, \textsl{ Finiteness properties of cubulated
  groups},
\newblock Compos. Math. \textbf{ 150}(3), 453--506 (2014).

\bibitem[LOS12]{LinnellOkunSchickStrongAtiyah}
P.~Linnell, B.~Okun and T.~Schick, \textsl{ The strong {A}tiyah conjecture for
  right-angled Artin and Coxeter groups},
\newblock Geom. Dedicata \textbf{ 158}, 261--266 (2012).

\bibitem[LS77]{LyndonSchupp}
R.~C. Lyndon and P.~E. Schupp,
\newblock \textsl{ Combinatorial group theory},
\newblock Springer-Verlag, Berlin, 1977.

\bibitem[Mar14a]{MartinCombinationEG}
A.~Martin, \textsl{ Combination of universal spaces for proper actions},
\newblock J. Homotopy Relat. Struct., in press  (2014).

\bibitem[Mar14b]{MartinBoundaries}
A.~Martin, \textsl{ Non-positively curved complexes of groups and boundaries},
\newblock Geom. Topol. \textbf{ 18}(1), 31--102 (2014).

\bibitem[MK14]{KourovkaProblems2014}
V.~Mazurov and E.~Khukro, editors,
\newblock \textsl{ The Kourovka notebook}, volume~18,
\newblock 2014.

\bibitem[MS71]{miller_embeddings_1971}
C.~F. Miller, III and P.~E. Schupp, \textsl{ Embeddings into {H}opfian groups},
\newblock J. Algebra \textbf{ 17}, 171--176 (1971).

\bibitem[MW02]{McCammondWiseFansLadders}
J.~P. McCammond and D.~T. Wise, \textsl{ Fans and ladders in small cancellation
  theory},
\newblock Proc. London Math. Soc. (3) \textbf{ 84}(3), 599--644 (2002).

\bibitem[Nic04]{nica_cubulating_2004}
B.~Nica, \textsl{ Cubulating spaces with walls},
\newblock Algebr. Geom. Topol. \textbf{ 4}, 297--309 (electronic) (2004).

\bibitem[Ol{\cprime}91]{olshanskii_geometry_1991}
A.~Y. Ol{\cprime}shanski{\u\i},
\newblock \textsl{ Geometry of defining relations in groups}, volume~70 of
  \textsl{ Mathematics and its {A}pplications ({S}oviet {S}eries)},
\newblock Kluwer {A}cademic {P}ublishers {G}roup, {D}ordrecht, 1991,
\newblock Translated from the 1989 {R}ussian original by {Y}u. {A}.
  {B}akhturin.

\bibitem[Pan99]{pankratev_hyperbolic_1999}
A.~E. Pankrat{\cprime}ev, \textsl{ Hyperbolic products of groups},
\newblock Vestnik Moskov. Univ. Ser. I Mat. Mekh. (2), 9--13, 72 (1999).

\bibitem[RS87]{RipsSegevUniqueProduct}
E.~Rips and Y.~Segev, \textsl{ Torsion-free group without unique product
  property},
\newblock J. Algebra \textbf{ 108}(1), 116--126 (1987).

\bibitem[Sag95]{SageevCubeComplex}
M.~Sageev, \textsl{ Ends of group pairs and non-positively curved cube
  complexes},
\newblock Proc. London Math. Soc. (3) \textbf{ 71}(3), 585--617 (1995).

\bibitem[Sch76]{schupp_embeddings_1976}
P.~E. Schupp, \textsl{ Embeddings into simple groups},
\newblock J. London Math. Soc. (2) \textbf{ 13}(1), 90--94 (1976).

\bibitem[Sch14]{SchreveStrongAtiyah}
K.~Schreve, \textsl{ The strong {A}tiyah conjecture for virtually cocompact
  special groups},
\newblock Math. Ann. \textbf{ 359}(3-4), 629--636 (2014).

\bibitem[Sta91]{GerstenStallings}
J.~R. Stallings,
\newblock Non-positively curved triangles of groups,
\newblock in \textsl{ Group theory from a geometrical viewpoint ({T}rieste,
  1990)}, pages 491--503, World Sci. Publ., River Edge, NJ, 1991.

\bibitem[Ste15]{SteenbockRipsSegev}
M.~Steenbock, \textsl{ {Rips--Segev torsion-free groups without the unique
  product property}},
\newblock J. Algebra \textbf{ 438}, 337--378 (2015),
\newblock http://www.sciencedirect.com/science/article/pii/S0021869315002343.

\bibitem[Wis04]{WiseSmallCancellation}
D.~T. Wise, \textsl{ Cubulating small cancellation groups},
\newblock Geom. Funct. Anal. \textbf{ 14}(1), 150--214 (2004).

\bibitem[Wis11]{wise_structure_2011}
D.~T. Wise, \textsl{ The Structure of Groups with a Quasiconvex Hierarchy},
\newblock (2011),
\newblock available online (accessed 08/18/2014),
\newblock
  \url{https://docs.google.com/file/d/0B45cNx80t5-2T0twUDFxVXRnQnc/edit?pli=1}.

\end{thebibliography}

\noindent Alexandre Martin, Fakult\"at f\"ur Mathematik, Oskar-Morgenstern-Platz 1, 1180 Wien, Austria.

\noindent E-mail: alexandre.martin@univie.ac.at\\

\noindent Markus Steenbock, Fakult\"at f\"ur Mathematik, Oskar-Morgenstern-Platz 1, 1180 Wien, Austria.

\noindent E-mail: markus.steenbock@univie.ac.at\\

\end{document}